\PassOptionsToPackage{frozencache}{minted}

\IfFileExists{./prepreamble-amspreprint.sty}{\RequirePackage[packages,theorems,changes]{prepreamble-amspreprint}}{}

\makeatletter
\providecommand*{\input@path}{}\edef\input@path{{./scoop-latex/}\input@path}
\@namedef{ver@minted.sty}{}
\@namedef{opt@minted.sty}{}
\makeatother
\PassOptionsToPackage{style = authoryear-comp, backend = biber}{biblatex}
\PassOptionsToPackage{commentmarkup = footnote}{changes}
\PassOptionsToPackage{ngerman, english}{babel}
\PassOptionsToPackage{autostyle = true}{csquotes}
\RequirePackage{algorithm}
\RequirePackage{algpseudocode}
\RequirePackage{pgfplotstable}    

\documentclass[]{amsart}

\RequirePackage{biblatex}
\RequirePackage{ltxcmds}
\IfFileExists{./preamble-amspreprint.sty}{\RequirePackage[packages,theorems,changes]{preamble-amspreprint}}{}

\usepackage{scoop-local}

\IfFileExists{./postpreamble-amspreprint.sty}{\RequirePackage[packages,theorems,changes]{postpreamble-amspreprint}}{}

\makeatletter
\@ifpackageloaded{changes}{
\definechangesauthor[name = {Bastian Dittrich}, color = {red!80!black}]{BD}
\definechangesauthor[name = {Evelyn Herberg}, color = {purple!80!black}]{EH}
\definechangesauthor[name = {Roland Herzog}, color = {orange!80!black}]{RH}
\definechangesauthor[name = {Georg Müller}, color = {teal}]{GM}
}{}
\makeatother        

\makeatletter
\@ifpackageloaded{hyperref}{%
	\hypersetup{
		pdftitle = {A DC-Reformulation for Gradient-\texorpdfstring{$L^0$}{L0}-Constrained Problems},
		pdfauthor = {Bastian Dittrich, Evelyn Herberg, Roland Herzog, Georg Müller},
		pdfkeywords = {\texorpdfstring{$L^0$}{L0} constraints, gradient sparsity, DC programming},
		pdfcreator = {Created using the Scoop Template Engine version 1.6.0.}
	}
}{
	\pdfinfo{
		/Title (A DC-Reformulation for Gradient-\texorpdfstring{$L^0$}{L0}-Constrained Problems)
		/Author (Bastian Dittrich, Evelyn Herberg, Roland Herzog, Georg Müller)
		/Subject ()
		/Keywords (\texorpdfstring{$L^0$}{L0} constraints, gradient sparsity, DC programming)
		/Creator (Created using the Scoop Template Engine version 1.6.0.)
	}
}
\makeatother

\title[DC-Reformulation for Gradient-\texorpdfstring{$L^0$}{L0}-Constrained Problems]{A DC-Reformulation for Gradient-\texorpdfstring{$L^0$}{L0}-Constrained Problems}

\author[B. Dittrich]{Bastian Dittrich}
\address[B. Dittrich]{Würzburg University, Emil-Fischer-Straße 30, 97074 Würzburg, Germany}
\email{\detokenize{bastian.dittrich@uni-wuerzburg.de}}
\urladdr{https://www.mathematik.uni-wuerzburg.de/optimalcontrol/team/dittrich-bastian/}

\author[E. Herberg]{Evelyn Herberg\orcidlink{0000-0003-2515-4818}}
\address[E. Herberg]{Interdisciplinary Center for Scientific Computing, Heidelberg University, 69120 Heidelberg, Germany}
\email{\detokenize{evelyn.herberg@iwr.uni-heidelberg.de}}
\urladdr{https://scoop.iwr.uni-heidelberg.de}

\author[R. Herzog]{Roland Herzog\orcidlink{0000-0003-2164-6575}}
\address[R. Herzog]{Interdisciplinary Center for Scientific Computing, Heidelberg University, 69120 Heidelberg, Germany and Institute for Mathematics, Heidelberg University, 69120 Heidelberg, Germany}
\email{\detokenize{roland.herzog@iwr.uni-heidelberg.de}}
\urladdr{https://scoop.iwr.uni-heidelberg.de}

\author[G. Müller]{Georg Müller\orcidlink{0000-0003-2515-4818}}
\address[G. Müller]{Interdisciplinary Center for Scientific Computing, Heidelberg University, 69120 Heidelberg, Germany}
\email{\detokenize{georg.mueller@iwr.uni-heidelberg.de}}
\urladdr{https://scoop.iwr.uni-heidelberg.de}

\thanks{The first author was partially supported by the German Research Foundation DFG under project grant Wa~3626/3--2.}

\date{\today}

\dedicatory{}

\begin{document}

\begin{abstract}
Cardinality constraints in optimization are commonly of $L^0$-type, and they lead to sparsely supported optimizers.
An efficient way of dealing with these constraints algorithmically, when the objective functional is convex, is reformulating the constraint using the difference of suitable $L^1$- and largest-$K$-norms and subsequently solving a sequence of penalized subproblems in the difference-of-convex (DC) class.
We extend this DC-reformulation approach to problems with $L^0$-type cardinality constraints on the support of the gradients, \ie, problems where sparsity of the gradient and thus piecewise constant solutions are the target.
We apply the resulting algorithm to a classical energy minimization problem, an optimal control problem, and an image reconstruction problem.

\end{abstract}


\makeatletter
\ltx@ifpackageloaded{hyperref}{%
\subjclass[2010]{\href{https://mathscinet.ams.org/msc/msc2020.html?t=49K27}{49K27}, \href{https://mathscinet.ams.org/msc/msc2020.html?t=49M20}{49M20}, \href{https://mathscinet.ams.org/msc/msc2020.html?t=90C26}{90C26}, \href{https://mathscinet.ams.org/msc/msc2020.html?t=90C46}{90C46}}
}{%
}
\makeatother

\maketitle

\section{Introduction}
\label{section:introduction}

A wide range of applications benefit from sparsity of the solutions to their characterizing optimization problems, \ie, when the solutions are localized in a suitable sense.
Standard example applications include sensor and control device placement problems \cite{Stadler:2009:1,ChepuriLeus:2015:1}, optimal experimental design \cite{AlexanderianPetraStadlerGhattas:2014:1,HaberMagnantLuceroTenorio:2012:1}, imaging problems \cite{GehreKluthLipponenJinSeppaenenKaipioMaass:2012:1,ZhongQin:2016:1}, optimization in finance \cite{WuSunGeAllenZhaoZeng:2024:1,BertsimasCoryWright:2022:1,BurkhardtUlrych:2023:1}, signal processing \cite{MarvastiAminiHaddadiSoltanolkotabiKhalajAldroubiSaneiChambers:2012:1} and machine learning \cite{DeleuBengio:2021:1,KuruogluKuoChan:2023:1}.
Common approaches for promoting sparsity in the solutions to optimization or optimal control problems proceed by adding corresponding constraints.
This can be done either explicitly by adding an upper bound on a suitable (semi"~)""norm, or implicitly by adding a penalty term to the objective function.

In \cite{GotohTakedaTono:2017:1}, the authors considered a cardinality-constrained problem in finite dimensions of the type
\begin{align*}
	\text{Minimize}
	\quad
	&
	f(x)
	\quad
	\text{where }
	x \in \R^n
	\\
	\text{\st}
	\quad
	&
	\norm{x}_0 \le K
	,
\end{align*}
where $\norm{\,\cdot\,}_0$ denotes the $\ell_0$-pseudo-norm that counts the non-zero elements of~$x$.
Their solution approach was based on solving a family of penalized problems, whose objective is the difference of convex (DC) functions, using a DC-specific algorithm; see, \eg, \cite{LeThiPham:2018:1}.
This approach has been extended in \cite{DittrichWachsmuth:2025:1} to infinite-dimensional problems of the type
\begin{equation}
	\begin{aligned}
		\label{eq:general-problem:function-value-constraints}
		\text{Minimize}
		\quad
		&
		f(u)
		\quad
		\text{where }
		u \in U
		\\
		\text{\st}
		\quad
		&
		\norm{u}_0 \le K
		,
	\end{aligned}
\end{equation}
in a reflexive Banach space~$U$ with compact embedding $U \compactly L^1(\Omega)$.
The $L^0$-pseudo-norm of a function $u \in U$ is defined as the measure of its support, \ie, $\norm{u}_0 = \mu(\supp(u))$, with respect to a suitable measure~$\mu$.

It is a key observation both in the finite- and infinite-dimensional settings that we may replace the $L^0$-pseudo-norm constraint leveraging the equivalence
\begin{equation}
	\label{eq:L0-norm-constraint}
	\norm{u}_0
	\le
	K
	\quad
	\Leftrightarrow
	\quad
	\norm{u}_1 - \largestKnorm{u}{K}
	=
	0
	.
\end{equation}
Here $\norm{\,\cdot\,}_1$ denotes the $L^1$-norm, and $\largestKnorm{\cdot}{K}$ the largest-$K$-norm.
We obtain a problem equivalent to \eqref{eq:general-problem:function-value-constraints} of the form
\begin{equation}
	\label{eq:general-problem:function-value-constraints:largest-K-norm}
	\begin{aligned}
		\text{Minimize}
		\quad
		&
		f(u)
		\quad
		\text{where }
		u \in U
		\\
		\text{\st}
		\quad
		&
		\norm{u}_1 - \largestKnorm{u}{K}
		=
		0
		,
	\end{aligned}
\end{equation}
for which the family of penalized, unconstrained problems with parameter $\rho > 0$
\begin{equation}
	\label{eq:general-problem:function-value-constraints:largest-K-norm:penalized}
	\text{Minimize}
	\quad
	f(u)
	+
	\rho \, \paren[big](){\norm{u}_1 - \largestKnorm{u}{K}}
	\quad
	\text{where }
	u \in U
	,
\end{equation}
yields approximations of the true solution, in the sense that every weak accumulation point of a sequence of global solutions to the penalized problems corresponding to $\rho \to \infty$ is a global solution to \eqref{eq:general-problem:function-value-constraints}, see \cite[Theorem~5.4]{DittrichWachsmuth:2025:1}.
Due to their structure featuring the difference of convex functions, the penalized problems \eqref{eq:general-problem:function-value-constraints:largest-K-norm:penalized} can be treated using the DC~algorithm.

\subsection*{Contributions}

In this paper, we present a modification of problem \eqref{eq:general-problem:function-value-constraints:largest-K-norm} and its penalized DC formulation \eqref{eq:general-problem:function-value-constraints:largest-K-norm:penalized} that is tailored to optimization problems with cardinality constraints on the (discretized) \emph{gradient} of the optimization function.
That is, we aim to achieve sparsity of the gradient rather than sparsity of the function itself.
Specifically, we are motivated by problems of the form
\begin{equation}
	\label{eq:general-problem:gradient-constraints}
	\begin{aligned}
		\text{Minimize}
		\quad
		&
		f(u)
		\quad
		\text{where }
		u \in U
		\\
		\text{\st}
		\quad
		&
		\norm{\nabla u}_0 \le K
		,
	\end{aligned}
\end{equation}
in a suitable function space~$U$ so that $u \in U$ has a weak gradient $\nabla u$.
The definition of the $L^0$-pseudo-norm of $\nabla u$, measuring the size of its support, will be given in the beginning of \cref{section:DC-reformulation:function-space}.

An upper bound on the measure of the gradient's support translates to sparsely supported gradients and thus to solutions that are constant on subsets.
This behavior is desirable, \eg, in imaging and signal reconstruction \cite{ArvanitopoulosDarginis:2017:1}.
Related literature can be found on elliptic equations with gradient constraints \cite{Evans:1979:1,GriesseKunisch:2009:1} and pointwise state gradient constraints in optimal control \cite{CasasFernandez:1993:1,DengMehlitzPruefert:2019:1,SchielaWollner:2011:1}.
We also mention \cite{ClasonKunisch:2014:1} for optimal control problems promoting piecewise constant controls, but with preferred values of the control.

Both the analytical properties and possible choices for solution algorithms for these types of problems differ significantly from those of the original problem formulation in \cite{DittrichWachsmuth:2025:1}, where sparsity of the function~$u$ itself was targeted.
On the analytical side, this results for example in non-trivial discrete optimality conditions due to the constraint acting on the interplay of nodal values instead of the nodal values itself.
On the numerical side, simply replacing function values with gradients would lead to singular matrices in a second-order subroutine, \eg, semismooth Newton, for solving the DC subproblems.
Therefore, a different approach to the discrete subproblems is necessary compared to \cite{DittrichWachsmuth:2025:1}.
Moreover, the non-smooth penalty term in the associated discrete DC subproblem would not be separable anymore and hence associated proximal operators do not have a closed form solution.
Therefore, we propose an approach which often results in quadratic subproblems.
This reformulation renders the associated penalization approach inexact, yet still effective.

The described procedure additionally suggests a novel way to solve problems stemming from Potts models in imaging applications where the $L^0$-regularizer is treated as an explicit side constraint, \cf \cite{StorathWeinmannFrikelUnser:2015:1} and the references therein.
Comparisons with state of the art techniques for this application are beyond the scope of this paper.

\subsection*{Structure}

\Cref{section:notation-preliminaries} introduces the required notation and preliminaries.
In \cref{section:DC-reformulation:function-space}, we extend analytical results of DC reformulations for cardinality-constrained problems in function space to reformulations of gradient-sparsity-type.
We establish existence of solution for the gradient sparsity problem and the viability of the reformulation approach.
\Cref{section:DC-reformulation:discrete} discusses the details of the discretization process and the analysis of the corresponding discrete problems including nodal Bouligand-stationarity conditions for the discretized system.
Finally \cref{section:numerical-results} contains numerical experiments.
\Cref{section:notation} collects the notation used throughout the paper in a helpful table.

\section{Notation and Preliminaries}
\label{section:notation-preliminaries}

The support and the $\ell_0$-pseudo-norm of a vector $x \in \R^n$ are defined as
\begin{equation*}
	\supp(x)
	\coloneqq
	\setDef[big]{i \in \set{1, \ldots, n}}{x_i \neq 0}
	,
	\qquad
	\norm{x}_0 \coloneqq \mu(\supp(x))
	.
\end{equation*}
Here $\mu \colon \cP(\set{1, \ldots, n}) \to \R$ denotes the counting measure, so that we have $\norm{x}_0 = \cardinality{\supp(x)}$.
Furthermore, we define a weighted version with weight vector $\nu \in \R_{> 0}^n$ as follows:
\begin{equation*}
	\norm{x}_{0,\nu}
	\coloneqq
	\mu_\nu(\supp(x))
	\coloneqq
	\sum_{i=1}^n \nu_i \, \norm{x_i}_0
	=
	\sum_{\substack{i=1 \\ x_i \neq 0}}^n \nu_i
	.
\end{equation*}
More generally, for measurable functions $u \colon \Omega \to \R$ on a measure space $(\Omega, \cA, \mu)$, the (set-theoretic) support and corresponding $L^0$-pseudo-norm are defined as
\begin{equation*}
	\supp(u)
	\coloneqq
	\setDef{x \in \Omega}{u(x) \neq 0}
	,
	\qquad
	\norm{u}_0 \coloneqq \mu(\supp(u))
	.
\end{equation*}

The reformulation of $L^0$-constrained problems into DC problems makes heavy use of the largest-$K$-norm.
When $x \in \R^n$, $\nu \in \R_{> 0}^n$ is a weight vector and $K \in \interval[][]{0}{\norm{\nu}_1}$, then the corresponding weighted largest-$K$-norm is defined as
\begin{equation*}
	\largestKnorm{x}{K}[\nu]
	\coloneqq
	\max \setDef[Big]{\sum_{i \in I} \nu_i \, \abs{x_i}}{I \subseteq \set{1, \ldots, n} \text{ \st\ } \sum_{i \in I} \nu_i \le K}
	.
\end{equation*}
This is a generalization of the definition of the usual (unweighted) largest-$K$-norm, which is equal to the sum of the absolute values of the $K$ largest entries of $x$ in absolute value.
The unweighted is recovered from the weighted case for $\nu = (1, \ldots, 1)^\transp$ and $K \in \N_0$.
When $K = \norm{\nu}_1$, then $\largestKnorm{x}{K}[\nu]$ becomes the weighted $\ell_1$-norm, \ie,
\begin{equation*}
	\norm{x}_{1,\nu}
	\coloneqq
	\sum_{i=1}^n \nu_i \, \abs{x_i}
	.
\end{equation*}
More generally, consider an integrable function $u \colon \Omega \to \R$ on an underlying measure space $(\Omega, \cA, \mu)$ and $K \in \interval[][]{0}{\mu(\Omega)}$.
Its largest-$K$-norm is defined as
\begin{equation*}
	\largestKnorm{u}{K}
	\coloneqq
	\sup \setDef[Big]{\int_A \abs{u} \d \mu}{A \in \cA \text{ \st\ } \mu(A) \le K}
	.
\end{equation*}
This generalization was extensively studied in \cite{DittrichWachsmuth:2025:1}.
Throughout the paper, we adopt the convention that $(+\infty) - (+\infty) = +\infty$, which is reasonable in the context of DC functions.

For any Banach space $X$, we let $X^*$ denote its topological dual space.
Additionally, for a linear, bounded mapping $g \colon X \to Y$ between Banach spaces, we let $g^*\colon Y^* \to X^*$ denote the dual mapping.
When $g \colon X \to \overline{\R} \coloneqq \R \cup \set{+\infty}$ is convex, we consider the convex subdifferential at $x \in X$ to be
\begin{equation*}
	\partial g(x)
	=
	\setDef{s \in X^*}{g(x) + \dual{s}{\hat x - x}_{X^*,X} \le g(\hat x) \text{ for all } \hat x \in X}
	.
\end{equation*}

\section{DC Reformulation for Gradient Sparsity in Function Space}
\label{section:DC-reformulation:function-space}

In this section we consider the abstract problem
\begin{equation}
	\label{eq:general-problem}
	\begin{aligned}
		\text{Minimize}
		\quad
		&
		f(u)
		\quad
		\text{where }
		u \in U
		\\
		\text{\st}
		\quad
		&
		\norm{\nabla u}_0 \le K
		,
	\end{aligned}
\end{equation}
with $K \in \interval[][]{0}{\mu(\Omega)}$ under the following standing assumptions.

\begin{assumption}[Standing assumptions for the continuous setting] \skipline
	\label{assumption:standing-assumptions}
	\begin{enumerate}
		\item
			\label{assumption:standing-assumptions:Omega}
			$\Omega \subseteq \R^d$ is open and bounded, $d \in \N$, and $\mu$ is the Lebesgue measure.

		\item
			\label{assumption:standing-assumptions:U}
			$U$ is a reflexive Banach space that is compactly embedded in $H^1(\Omega)$ and is dense in $L^1(\Omega)$.

		\item
			\label{assumption:standing-assumptions:f}
			The objective $f \colon U \to \overline{\R} \coloneqq \R \cup \set{+\infty}$ is proper, convex, lower semicontinuous, bounded from below and radially unbounded.

		\item
			\label{assumption:standing-assumptions:extraproper}
			There exists $u \in U$ with $f(u) < \infty$ and $\norm{\nabla u}_0 \le K$.
	\end{enumerate}
\end{assumption}

\begin{example}
	\label{example:standing-assumptions}
	Suppose that $\Omega \subseteq \R^d$, $d \ge 2$ is a bounded domain with Lipschitz boundary.
	For any $\varepsilon > 0$, an example of a space $U$ that satisfies \cref{assumption:standing-assumptions:U} of \cref{assumption:standing-assumptions} is $U \coloneqq H^1_0(\Omega) \cap H^{1+\varepsilon}(\Omega)$ equipped with the $H^{1+\varepsilon}(\Omega)$-norm.
\end{example}

Let us discuss the precise meaning of $\norm{\nabla u}_0$.
Here and throughout, $\frac{\partial u}{\partial x_i}$ denotes the $i$-th weak partial derivative of~$u$.
By $\nabla u \coloneqq \paren[big](){\frac{\partial u}{\partial x_1}, \ldots, \frac{\partial u}{\partial x_d}}^\transp$, we denote the Euclidian weak gradient of $u \in U$.
We define
\begin{equation}
	\label{eq:support-of-gradient}
	\supp(\nabla u)
	\coloneqq
	\bigcup_{i=1}^d \supp \paren[Big](){\frac{\partial u}{\partial x_i}}
	=
	\supp \paren[Big](){\sum_{i=1}^d \paren[big](){\textstyle\frac{\partial u}{\partial x_i}}^2}
	,
\end{equation}
and consequently,
\begin{math}
	\norm{\nabla u}_0
	\coloneqq
	\mu(\supp(\nabla u))
\end{math}
satisfies
\begin{equation*}
	\norm{\nabla u}_0
	=
	\norm{(\nabla u)^\transp \nabla u}_0
	=
	\norm[Big]{\sum_{i=1}^d \paren[big](){\textstyle\frac{\partial u}{\partial x_i}}^2}_0
	=
	\norm{\sumOfSquaredGradComponents{u}}_0
	.
\end{equation*}
Notice that this definition is invariant \wrt rotation of the coordinate system.
We also have
\begin{equation*}
	\largestKnorm{\sumOfSquaredGradComponents{u}}{K}
	=
	\sup \setDef[Big]{\int_A \sumOfSquaredGradComponents{u} \d \mu}{A \text{ is measurable and } \mu(A) \le K}
	.
\end{equation*}

Since $U$ is embedded into $H^1(\Omega)$, we know that $\sumOfSquaredGradComponents{u} \in L^1(\Omega)$ holds for all $u \in U$.
By \cite[Theorem~3.22]{DittrichWachsmuth:2025:1}, problem \eqref{eq:general-problem} is therefore equivalent to the problem
\begin{equation}
	\label{eq:general-problem:largest-K-norm}
	\begin{aligned}
		\text{Minimize}
		\quad
		&
		f(u)
		\quad
		\text{where }
		u \in U
		\\
		\text{\st}
		\quad
		&
		\norm{\sumOfSquaredGradComponents{u}}_1 - \largestKnorm{\sumOfSquaredGradComponents{u}}{K}
		=
		0
		.
	\end{aligned}
\end{equation}
We associate with \eqref{eq:general-problem:largest-K-norm} a family of unconstrained, penalized problems
\begin{equation}
	\label{eq:general-problem:penalized}
	\begin{aligned}
		\text{Minimize}
		\quad
		&
		f_\rho(u)
		\coloneqq
		f(u) + \rho \, \phi(u)
		\quad
		\text{where }
		u \in U
		\\
		\text{with}
		\quad
		&
		\phi(u)
		\coloneqq
		\norm{\sumOfSquaredGradComponents{u}}_1 - \largestKnorm{\sumOfSquaredGradComponents{u}}{K}
		,
	\end{aligned}
\end{equation}
and penalty parameter $\rho > 0$.
The family \eqref{eq:general-problem:penalized} can be used to approximate the solutions of \eqref{eq:general-problem}, as shown by the following results.

\begin{lemma}
	\label{lemma:weak-continuity-of-phi}
	Both $\norm{\sumOfSquaredGradComponents{(\cdot)}}_1$ and $\largestKnorm{\sumOfSquaredGradComponents{(\cdot)}}{K}$ are weakly sequentially continuous on~$U$, and therefore so is $\phi$.
\end{lemma}
\begin{proof}
	Suppose $\sequence{u}{k} \weakly u^*$ in~$U$.
	By the compact embedding $U \compactly H^1(\Omega)$, this convergence is even strong in $H^1(\Omega)$ and in particular we have $\frac{\partial \sequence{u}{k}}{\partial x_i} \to \frac{\partial u^*}{\partial x_i}$ in $L^2(\Omega)$ for all $i = 1, \ldots, d$.
	Therefore, $\sumOfSquaredGradComponents{\sequence{u}{k}} \to \sumOfSquaredGradComponents{u^*}$ in $L^1(\Omega)$, and the claim follows because $\largestKnorm{\cdot}{K} \le \norm{\,\cdot\,}_1$.
\end{proof}

\begin{theorem}
	\label{theorem:general-problem:penalized:existence}
	Problems \eqref{eq:general-problem}, \eqref{eq:general-problem:largest-K-norm} and \eqref{eq:general-problem:penalized} possess global minimizers.
\end{theorem}
\begin{proof}
	With \cref{lemma:weak-continuity-of-phi}, the result follows analogously as in \cite[Theorem~5.2 and Theorem~5.3]{DittrichWachsmuth:2025:1}.
\end{proof}

\begin{theorem}[Convergence of global minimizers] \skipline
	\label{theorem:general-problem:penalized:convergence}
	Suppose that $\sequence{\rho}{k} > 0$ is a sequence such that $\lim_{k \to \infty} \sequence{\rho}{k} = \infty$.
	Let $\sequence{u}{k}$ be a global solution of the penalized problem \eqref{eq:general-problem:penalized} with $\rho = \sequence{\rho}{k}$ for all $k \in \N$.
	Then the sequence $(\sequence{u}{k})$ is bounded in~$U$, satisfies $\lim_{k \to \infty} \phi(\sequence{u}{k}) = 0$, and every weak accumulation point $\overline{u}$ is a global solution of \eqref{eq:general-problem}.
\end{theorem}
\begin{proof}
	Together with \cref{lemma:weak-continuity-of-phi}, the result follows analogously as in \cite[Theorem~5.4]{DittrichWachsmuth:2025:1}.
\end{proof}

Since critical points are the natural limit points of the iterates in a DC~algorithm, we are interested in their characterization.
\begin{definition}[Critical point, strongly critical point] \skipline
	\label{definition:critical-point}
	Suppose that $g, h \colon U \to \overline{\R}$ are proper, lower semicontinuous and convex and consider the DC function $\delta \coloneqq g - h$.
	Then $u \in U$ is called a critical point of~$\delta$ \wrt the decomposition into $g$ and $h$ if
	\begin{equation*}
		\partial g(u) \cap \partial h(u)
		\neq
		\emptyset
		,
	\end{equation*}
	and a strongly critical point if
	\begin{equation*}
		\emptyset
		\neq
		\partial h(u)
		\subseteq
		\partial g(u)
		.
	\end{equation*}
\end{definition}

Note that the conditions in \cref{definition:critical-point} for (strong) criticality are necessary local optimality conditions for a DC function $\delta = g-h$, see \cite[Theorem~2]{PhamLeThi:1997:1}.
Under sufficient regularity assumptions, (strong) criticality of $u$ corresponds to the well known necessary condition $0 \in \partial (g - h)(u)$ for \emph{local} optimality.
We will go into more detail in the discretized setting in \cref{subsection:DC-reformulation:DC-approach}.
For now, applying the chain rule for the computation of convex subdifferentials, we can characterize the subdifferential of $\largestKnorm{\sumOfSquaredGradComponents{u}}{K}$.
To simplify notation, we introduce the function
\begin{equation*}
	W
	\colon
	U
	\ni
	u
	\mapsto
	W(u)
	\coloneqq
	\sumOfSquaredGradComponents{u}
	=
	\sum_{i=1}^d \paren[Big](){\frac{\partial u}{\partial x_i}}^2
	\in
	L^1(\Omega)
	,
\end{equation*}
that has the continuous Fréchet derivative
\begin{equation*}
	W'(u) \, v
	=
	2 \sum_{i=1}^d \frac{\partial u}{\partial x_i} \frac{\partial v}{\partial x_i}
	\quad
	\text{ for }
	v \in U
	.
\end{equation*}

\begin{theorem}[Characterization of the subdifferential $\partial \largestKnorm{W(\cdot)}{K}$] \skipline
	\label{theorem:subdifferential-of-the-largest-K-norm-of-the-sum-of-squared-gradient-components}
	Let $u \in U$ and $K \in \interval[][]{0}{\mu(\Omega)}$.
	Then $\largestKnorm{W(\cdot)}{K} \colon U \to \R$ is convex and
	\begin{subequations}
		\begin{equation}
			\partial \largestKnorm{W(\cdot)}{K}(u)
			=
			\setDef[auto]{W'(u)^* r}{r \in \partial \largestKnorm{\cdot}{K}(W(u))}
			,
		\end{equation}
		holds with
		\begin{multline}
			\label{eq:subdifferential-of-the-largest-K-norm}
			\partial \largestKnorm{\cdot}{K}\paren[auto](){W(u)}
			\\
			=
			\setDef[auto]{r \in L^\infty(\Omega)}{\norm{r}_\infty \le 1, \; \norm{r}_1 \le K, \; \int_\Omega r \, W(u) \d \mu = \largestKnorm{W(u)}{K}}
			.
		\end{multline}
	\end{subequations}
	This is to be understood in the sense of the embedding of $L^\infty(\Omega)$ into $U^*$ induced by the embedding of $U$ into $L^1(\Omega)$; see \cref{assumption:standing-assumptions}.
	Moreover,
	\begin{equation*}
		(W'(u)^* r)\, v
		=
		2 \sum_{i=1}^d \int_\Omega r \, \frac{\partial u}{\partial x_i} \frac{\partial v}{\partial x_i} \d \mu
		\quad
		\text{ for }
		v \in U
		.
	\end{equation*}
\end{theorem}
\begin{proof}
	The largest-$K$-norm is convex, continuous and Lipschitz-continuous everywhere and its convex subdifferential has the claimed form \eqref{eq:subdifferential-of-the-largest-K-norm}; see, \eg, \cite[Theorem~3.15]{DittrichWachsmuth:2025:1}.

	Notice that due to the definition $W(u) = \sum_{i=1}^d \paren[big](){\frac{\partial u}{\partial x_i}}^2$, we have that $W \colon U \to L^1(\Omega)$ is convex in a pointwise \ale-sense.
	That is, for $t \in \interval[][]{0}{1}$ and $u,v \in U$, we have $W(t \, u + (1 - t) \, v) \le t \, W(u) + (1 - t) \, W(v)$.
	Accordingly, because of nonnegativity of~$W$, we find
	\begin{align*}
		\MoveEqLeft
		\largestKnorm{W(t \, u + (1 - t) \, v)}{K}
		\\
		&
		=
		\sup \setDef[Big]{\int_A W(t \, u + (1 - t) \, v) \d \mu}{A \in \cA \text{ \st\ } \mu(A) \le K}
		\\
		&
		\le
		\sup \setDef[Big]{\int_A t \, W(u) \d \mu + \int_A (1 - t) \, W(v) \d \mu}{A \in \cA \text{ \st\ } \mu(A) \le K}
		\\
		&
		\le
		t \, \sup \setDef[Big]{\int_A W(u) \d \mu}{A \in \cA \text{ \st\ } \mu(A) \le K}
		\\
		&
		\quad
		+
		(1 - t) \, \sup \setDef[Big]{\int_A W(v) \d \mu}{A \in \cA \text{ \st\ } \mu(A) \le K}
		\\
		&
		=
		t \, \largestKnorm{W(u)}{K}
		+
		(1 - t) \, \largestKnorm{W(v)}{K}
		,
	\end{align*}
	showing the convexity of $\largestKnorm{W(\cdot)}{K} \colon U \to \R$.

	Since $\largestKnorm{W(\cdot)}{K}$ is weakly sequentially continuous on~$U$ and convex, see \cref{lemma:weak-continuity-of-phi}, by application of the chain rule for convex subdifferentials, see, \eg \cite[Theorems~13.23 and~13.8]{ClasonValkonen:2025:1}, the elements of $\partial \largestKnorm{W(\cdot)}{K}(u) = \partial(\largestKnorm{\cdot}{K} \circ W)(u)$ are of the form $W'(u)^* r$ with $r \in \partial \largestKnorm{\cdot}{K}(W(u))$ and
	\begin{equation*}
		(W'(u)^* r) \, v
		=
		\dual{W'(u)^* r}{v}_{U^*,U}
		=
		\dual{r}{W'(u) \, v}_{L^\infty(\Omega),L^1(\Omega)}
		=
		2 \sum_{i=1}^d \int_\Omega r \, \frac{\partial u}{\partial x_i} \frac{\partial v}{\partial x_i} \d \mu
	\end{equation*}
	holds for all $v \in U$.
\end{proof}

\Cref{theorem:subdifferential-of-the-largest-K-norm-of-the-sum-of-squared-gradient-components} allows us to state the following necessary optimality condition.
\begin{corollary}
	Let $f$ be Gâteaux-differentiable and let $u$ be a critical point of $f_\rho$ \wrt the natural DC decomposition in the DC problem \eqref{eq:general-problem:penalized}.
	Then there exists $r \in L^\infty(\Omega)$ with $\norm{r}_\infty \le 1$, $\norm{r}_1 \le K$ and $\int_\Omega r \, W(u) \d \mu = \largestKnorm{W(u)}{K}$ such that
	\begin{equation*}
		f'(u) \, v
		+
		2 \, \rho \sum_{i=1}^d \int_\Omega (1 - r) \frac{\partial u}{\partial x_i} \frac{\partial v}{\partial x_i} \d \mu
		=
		0
		\quad
		\text{ for all }
		v \in U
		.
	\end{equation*}
	If $u$ is even strongly critical, then the subdifferential $\partial \largestKnorm{W(\cdot)}{K}(u)$ is a singleton and the unique element $s \in \partial \largestKnorm{W(\cdot)}{K}(u)$ satisfies
	\begin{equation*}
		\rho \, \dual{s}{v}_{U^*,U}
		=
		f'(u) \, v
		+
		2 \, \rho \sum_{i=1}^d \int_\Omega \frac{\partial u}{\partial x_i} \frac{\partial v}{\partial x_i} \d \mu
		\quad
		\text{ for all }
		v \in U
		.
	\end{equation*}
\end{corollary}
\begin{proof}
	First off, note that the function $u \mapsto$ $\norm{W(u)}_1 = \textstyle \sum_{i=1}^d \int_{\Omega} \paren[big](){\frac{\partial u}{\partial x_i}}^2 \d \mu$ is Fréchet-differentiable.
	Therefore, the first identity is a straightforward combination of \cref{definition:critical-point} with $\delta = f_\rho$, $g = f + \rho \, \norm{W(u)}_1$ and $h = \rho \, \largestKnorm{W(u)}{K}$, and the subdifferential characterization in \cref{theorem:subdifferential-of-the-largest-K-norm-of-the-sum-of-squared-gradient-components}.

	In case of strong criticality, we know that $\partial g(u)$ is a singleton, hence $\partial g(u)$ has to coincide with any of its nonempty subsets, which yields the second equality.
\end{proof}

\section{DC Reformulation for Gradient Sparsity in a Discretized Setting}
\label{section:DC-reformulation:discrete}

In this section we develop and analyze a discretization for problem \eqref{eq:general-problem} and its variants.
While the discussion in the continuous setting in \cref{section:DC-reformulation:function-space} was following \cite{DittrichWachsmuth:2025:1} quite closely, the discrete settings are notably different.

\subsection{Discretization, Constraint Reformulation, Existence of Solutions}
\label{subsection:discretization-constraint-reformulation-existence}

For simplicity, we assume in this section that the bounded domain $\Omega \subseteq \R^d$ is polyhedral.
We consider a geometrically conforming discretization of~$\Omega$ into a finite number of $d$-simplices~$S_i$; see \cite[Definition~1.55]{ErnGuermond:2004:1}.
We refer to the collection of simplices of dimensions~$0$ (vertices), $1$ (edges) up to $d$ as the mesh~$\cS$.
In other words, the mesh is a pure, geometric simplicial $d$-complex in~$\R^d$.
We let $\nvertices$, $\nedges$ and $\nsimplices$ denote the number of vertices, edges and $d$-simplices (cells) in~$\cS$.
In the coming analysis, we additionally require some notation concerning the relationships of vertices, edges and simplices, including the following index maps that encode incidence (connectivity) information:
\begin{equation*}
	\begin{aligned}
		\edgetovertices{j}
		&
		\coloneqq
		\setDef[big]{\ell \in \set{1, \ldots, \nvertices}}{\text{vertex $\ell$ is incident to edge $j$}}
		,
		\\
		\simplextovertices{i}
		&
		\coloneqq
		\setDef[big]{\ell \in \set{1, \ldots, \nvertices}}{\text{vertex $\ell$ is a vertex of simplex $i$}}
		,
		\\
		\vertextoedges{\ell}
		&
		\coloneqq
		\setDef[big]{j \in \set{1, \ldots, \nedges}}{\text{edge $j$ is incident to vertex $\ell$}}
		,
		\\
		\simplextoedges{i}
		&
		\coloneqq
		\setDef[big]{j \in \set{1, \ldots, \nedges}}{\text{edge $j$ is an edge of simplex $i$}}
		,
		\\
		\vertextosimplices{\ell}
		&
		\coloneqq
		\setDef[big]{i \in \set{1, \ldots, \nsimplices}}{\text{simplex $i$ is incident to vertex $\ell$}}
		,
		\\
		\edgetosimplices{j}
		&
		\coloneqq
		\setDef[big]{i \in \set{1, \ldots, \nsimplices}}{\text{simplex $i$ is incident to edge $j$}}
		.
	\end{aligned}
\end{equation*}
When either of the maps above is applied to a set of arguments, the output is to be understood as the union of the images of the individual arguments.

We consider the finite element space of piecewise affine and continuous functions $\Omega \to \R$, \ie,
\begin{equation}
	\label{eq:finite-element-space}
	U_h
	\coloneqq
	\Span \setDef[big]{\varphi_\ell}{\ell \in \set{1, \ldots, \nvertices}}
	,
\end{equation}
generated by the nodal basis functions $\varphi_\ell$, $\ell = 1, \ldots, \nvertices$.
Elements $u_h \in U_h$ will be represented by their nodal coefficient vectors, which we denote by $\tu \in \R^\nvertices$.
Due to the geometrically conforming discretization, we have $U_h \subseteq H^s(\Omega)$ for all $s \in \interval[][){0}{3/2}$; see \cite[p.578]{Fusco:2015:1} or \cite[Corollary~2.2~(i)]{Lombardini:2018:1}.

In the following, we will be working with a proxy variable~$\tw$ that indicates which simplices a finite element function is constant on.
\begin{assumption}
	\label{assumption:abstract-w-properties}
	Given a weight vector $\nu \in \R^\nsimplices_{>0}$, we assume that $\tw \colon \R^\nvertices \to \R^\nsimplices$ is a function that satisfies
	\begin{enumerate}
		\item
			\label{assumption:abstract-w-properties:indicator}
			$\sumOfSquaredDifferences{\tu}_i = 0$ if and only if the respective  function $u_h \in U_h$ is constant on~$S_i$,

		\item
			\label{assumption:abstract-w-properties:continuity}
			$\tw$ is continuous,

		\item
			\label{assumption:abstract-w-properties:convexity}
			$\tu \mapsto \largestKnorm{\sumOfSquaredDifferences{\tu}}{K}[\nu]$ is convex.
	\end{enumerate}
\end{assumption}

Using $\tw$, we collect the indices of cells where $\tu$ is constant or non-constant in the sets
\begin{align*}
	\constantsimplices{\tu}
	&
	\coloneqq
	\setDef[big]{i \in \set{1, \ldots, \nsimplices}}{u_h \text{ is constant on } S_i}
	=
	\setDef[big]{i \in \set{1, \ldots, \nsimplices}}{\sumOfSquaredDifferences{\tu}_i = 0}
	,
	\\
	\nonconstantsimplices{\tu}
	&
	\coloneqq
	\set{1, \ldots, \nsimplices} \setminus \constantsimplices{\tu}
	=
	\setDef[big]{i \in \set{1, \ldots, \nsimplices}}{\sumOfSquaredDifferences{\tu}_i \neq 0}
	,
\end{align*}
where both equalities hold due to \cref{assumption:abstract-w-properties}~\ref{assumption:abstract-w-properties:indicator}.

Given a weight vector $\nu \in \R^\nsimplices_{>0}$ and $\tu \in \R^\nvertices$, for $\sumOfSquaredDifferences{\tu}$, the weighted (pseudo"~)""norms defined in \cref{section:notation-preliminaries} have the form
\begin{align*}
	\norm{\sumOfSquaredDifferences{\tu}}_{0,\nu}
	&
	=
	\sum_{\mrep{i \in \nonconstantsimplices{\tu}}{}} \nu_i
	,
	\\
	\norm{\sumOfSquaredDifferences{\tu}}_{1,\nu}
	&
	=
	\sum_{i=1}^\nsimplices \nu_i \, \abs{\sumOfSquaredDifferences{\tu}_i}
	,
	\\
	\largestKnorm{\sumOfSquaredDifferences{\tu}}{K}[\nu]
	&
	=
	\max \setDef[Big]{\sum_{i \in I \cap \nonconstantsimplices{\tu}} \nu_i \, \abs{\sumOfSquaredDifferences{\tu}_i} }{I \subseteq \set{1, \ldots, \nsimplices} \text{ \st\ } \sum_{i \in I} \nu_i \le K}
	.
\end{align*}

Unless otherwise stated, we always use the weight vector $\nu$ given by the cell volumes
\begin{equation}
	\label{eq:cell-volumes-are-weights}
	\nu_i
	\coloneqq
	\mu(S_i)
	\text{ for }
	i = 1, \ldots, \nsimplices
	,
	\text{ so that }
	\norm{\nu}_1 = \mu(\Omega)
	\text{ holds}
	.
\end{equation}
This allows us to equivalently express the cardinality constraint on $\nabla u_h$ as follows:
\begin{lemma}
	\label{lemma:L0-norm-of-gradient}
	For $u_h \in U_h$, we have $\norm{\nabla u_h}_0 = \norm{\sumOfSquaredDifferences{\tu}}_{0,\nu}$.
\end{lemma}
\begin{proof}
	Since $u_h$ is piecewise affine, $\nabla u_h$ is piecewise constant.
	Hence $\supp(\nabla u_h)$ is the union of entire cells~$S_i$; see \eqref{eq:support-of-gradient}.
	Consider an arbitrary cell~$S_i$.
	Since $u_h$ is affine, $\sumOfSquaredGradComponents{u_h}$ is constant on~$S_i$.
	In particular, $\sumOfSquaredGradComponents{u_h} = 0$ holds on~$S_i$ if and only if the vertex values of~$u_h$ on all vertices of~$S_i$ coincide, \ie, $i \in \constantsimplices{\tu}$.
	By \cref{assumption:abstract-w-properties}, this is the case if and only if $\sumOfSquaredDifferences{\tu}_i = 0$.
	Since different cells intersect only in sets of measure zero, the result follows.
\end{proof}

We now consider the discretized counterpart of \eqref{eq:general-problem}:
\begin{equation*}
	\begin{aligned}
		\text{Minimize}
		\quad
		&
		f(u_h)
		\quad
		\text{where }
		u_h \in U_h
		\\
		\text{\st}
		\quad
		&
		\norm{\nabla u_h}_0 \le K
		,
	\end{aligned}
\end{equation*}
with $K \in \interval[][]{0}{\mu(\Omega)}$.
Using \cref{lemma:L0-norm-of-gradient}, we can equivalently express this problem in terms of the cofficient vector~$\tu \in \R^\nvertices$ of $u_h \in U_h$:
\begin{equation}
	\label{eq:general-problem:discrete}
	\begin{aligned}
		\text{Minimize}
		\quad
		&
		f(\tu)
		\coloneqq
		f(u_h)
		\quad
		\text{where }
		\tu \in \R^\nvertices
		\\
		\text{\st}
		\quad
		&
		\norm{\sumOfSquaredDifferences{\tu}}_{0,\nu} \le K
		.
	\end{aligned}
\end{equation}
The slight abuse of notation of re-using the symbol~$f$ will not be harmful since $\tu$ and $u_h$ are in a bijective linear correspondence.
For the discretized setting, we replace \cref{assumption:standing-assumptions} by the following:
\begin{assumption}[Standing assumptions for the discrete setting] \skipline
	\label{assumption:standing-assumptions:finite-dimensions}
	The objective function $f \colon \R^\nvertices \to \overline{\R}$ is proper, lower semicontinuous and radially unbounded.
\end{assumption}

Under these assumptions on~$f$ we can prove the existence of minimizers for problem \eqref{eq:general-problem:discrete}.
However, we first need to show the closedness of the feasible set
\begin{equation}
	\label{eq:general-problem:discrete:feasible-set}
	F
	\coloneqq
	\setDef[big]{\tu \in \R^\nvertices}{\norm{\sumOfSquaredDifferences{\tu}}_{0,{\nu}} \le K}
	.
\end{equation}
\begin{lemma}
	\label{lemma:convergent-nodal-values-detect-non-constant-cells}
	Suppose that $\sequence[big](){\tu}{k} \subseteq \R^\nvertices$ satisfies $\sequence{\tu}{k} \to \tu$.
	Then there exists $N \in \N$ such that
	\begin{equation*}
		\nonconstantsimplices{\tu}
		\subseteq
		\nonconstantsimplices{\sequence{\tu}{k}}
		\quad
		\text{for all }
		k \ge N
		.
	\end{equation*}
	In particular, if $\norm{\sumOfSquaredDifferences{\sequence{\tu}{k}}}_{0,\nu} \le K$, then $\norm{\sumOfSquaredDifferences{\tu}}_{0,\nu} \le K$.
\end{lemma}
\begin{proof}
	Let $\varepsilon \coloneqq \frac{1}{2} \min \setDef[big]{\abs{\sumOfSquaredDifferences{\tu}_i}}{i \in \nonconstantsimplices{\tu}}$.
	Then, by \cref{assumption:abstract-w-properties}~\ref{assumption:abstract-w-properties:continuity}, there exists $N \in \N$ with $\norm{\sumOfSquaredDifferences{\sequence{\tu}{k}} - \sumOfSquaredDifferences{\tu}}_\infty \le \varepsilon$ for all $k \ge N$.
	This particularly implies $\sumOfSquaredDifferences{\sequence{\tu}{k}}_i \neq 0$ for all $i \in \nonconstantsimplices{\tu}$ and all $k \ge N$.
\end{proof}
Note that this result also shows that cells where $u_h$ is not constant will eventually be detected when the nodal values converge.
\begin{theorem}
	\label{theorem:general-problem:discrete:existence}
	Problem \eqref{eq:general-problem:discrete} possesses a global minimizer.
\end{theorem}
\begin{proof}
	The feasible set $F$ \eqref{eq:general-problem:discrete:feasible-set} is nonempty and closed, as shown by $\tnull \in F$ and \cref{lemma:convergent-nodal-values-detect-non-constant-cells}.
	We denote its characteristic function (with values in $\set{0,\infty}$) by $\characteristicFunction{F}$.
	Therefore and by the assumptions, $f + \characteristicFunction{F}$ is proper, lower semicontinuous and radially unbounded.
	Hence, problem \eqref{eq:general-problem:discrete} is solvable by Weierstraß' theorem; see, \eg, \cite[Theorem~1.14]{DharaDutta:2011:1}.
\end{proof}

Similarly as in \eqref{eq:general-problem:largest-K-norm}, we may also restate the cardinality constraint in \eqref{eq:general-problem:discrete} using a difference of weighted norms:
\begin{theorem}
	\label{theorem:L0-norm-gradient-constraint}
	Suppose $u_h \in U_h$ and $K \in \interval[][]{0}{\mu(\Omega)}$.
	Then we have
	\begin{equation*}
		\norm{\nabla u_h}_0
		\le
		K
		\quad
		\Leftrightarrow
		\quad
		\norm{\sumOfSquaredDifferences{\tu}}_{0,\nu}
		\le
		K
		\quad
		\Leftrightarrow
		\quad
		\norm{\sumOfSquaredDifferences{\tu}}_{1,\nu} - \largestKnorm{\sumOfSquaredDifferences{\tu}}{K}[\nu]
		=
		0
		.
	\end{equation*}
\end{theorem}
\begin{proof}
	This follows from \cref{lemma:L0-norm-of-gradient} and \cite[Theorem~3.22]{DittrichWachsmuth:2025:1}.
	The cases $K = 0$ and $K = \mu(\Omega)$ are also valid, as one can verify using the exact same arguments.
\end{proof}

Defining
\begin{equation}
	\label{eq:general-problem:discrete:penalty}
	\phi(\tu)
	\coloneqq
	\norm{\sumOfSquaredDifferences{\tu}}_{1, \nu}
	-
	\largestKnorm{\sumOfSquaredDifferences{\tu}}{K}[\nu]
	,
\end{equation}
we can restate the discrete problem \eqref{eq:general-problem:discrete} equivalently as
\begin{equation}
	\label{eq:general-problem:discrete:largest-K-norm}
	\begin{aligned}
		\text{Minimize}
		\quad
		&
		f(\tu)
		\quad
		\text{where }
		\tu \in \R^\nvertices
		\\
		\text{\st}
		\quad
		&
		\phi(\tu)
		=
		0
		,
	\end{aligned}
\end{equation}
and the associated penalized problem as
\begin{equation}
	\label{eq:general-problem:discrete:penalized}
	\text{Minimize}
	\quad
	f_\rho(\tu)
	\coloneqq
	f(\tu)
	+
	\rho \, \phi(\tu)
	\quad
	\text{where }
	\tu \in \R^\nvertices
	,
\end{equation}
with penalty parameter $\rho > 0$.

We will discuss optimality conditions for the discrete problem \eqref{eq:general-problem:discrete} (or, equivalently, \eqref{eq:general-problem:discrete:largest-K-norm}) in the following section and address algorithmic treatment of the penalized problem \eqref{eq:general-problem:discrete:penalized} in \cref{subsection:DC-reformulation:DC-approach}.

\subsection{Optimality Conditions}
\label{subsection:optimality-conditions}

In this subsection, we derive necessary optimality conditions of Bouligand-type for minimizers of the discrete optimization problem \eqref{eq:general-problem:discrete}.
While these optimality conditions are generally too strong for the limit points of the DC Algorithm to satisfy, as we will see in \cref{subsection:DC-reformulation:DC-approach}, they are well suited to reveal the intuitive nodal conditions on optimality.
For this subsection, we assume the following from now on:
\begin{assumption}[Additional standing assumptions for the discrete setting] \skipline
	\label{assumption:standing-assumptions:finite-dimensions:differentiability}
	In addition to \cref{assumption:standing-assumptions:finite-dimensions}, the objective function~$f$ is Hadamard directionally differentiable.
\end{assumption}
We begin by considering the tangent cone
\begin{equation*}
	\tangentCone{F}{\tu}
	\coloneqq
	\setDef[Big]
	{\td \in \R^\nvertices}
	{\exists \sequence[big](){\tu}{k} \subseteq F \text{ \st\ } \sequence{\tu}{k} \to \tu, \; \sequence{\tau}{k} \searrow 0, \; \td = \lim_{k \to \infty} \frac{\sequence{\tu}{k} - \tu}{\sequence{\tau}{k}}}
\end{equation*}
for any feasible $\tu \in F$; see \eqref{eq:general-problem:discrete:feasible-set}.
We say that $\tu \in F$ is a Bouligand-stationary (B-stationary) point in case
\begin{equation}
	\label{eq:B-stationarity}
	f'(\tu) \, \td
	\ge
	0
	\quad
	\text{for all }
	\td \in \tangentCone{F}{\tu}
\end{equation}
holds.
Our main result is \cref{theorem:B-stationarity:explicit-description} --- a nodal characterization of B-stationarity via vanishing averages of derivatives of the objective functional on certain connected subsets of the vertices.

Before proving the main result, we first need to establish some additional relations between the nodal coefficient vector $\tu \in \R^\nvertices$ and the proxy quantity vector $\sumOfSquaredDifferences{\tu}$.
The following result shows that the set of non-constant cells may expand under summation:
\begin{lemma}
	\label{Lem:supp}
	Suppose $\tu, \tv \in \R^\nvertices$.
	Then
	\begin{equation*}
		\nonconstantsimplices{\tu + \tv}
		\subseteq
		\nonconstantsimplices{\tu}
		\cup
		\nonconstantsimplices{\tv}
		.
	\end{equation*}
\end{lemma}
\begin{proof}
	Suppose $i \in \nonconstantsimplices{\tu + \tv}$, \ie, the corresponding $u_h + v_h$ is non-constant on the simplex with index~$i$.
	Of course this can only be the case if $u_h$ and $v_h$ are not both constant on that simplex at the same time, so $i \in \nonconstantsimplices{\tu} \cup \nonconstantsimplices{\tv}$.
\end{proof}

We can now proceed to a description of the tangent cone to the feasible set~$F$ depending on the proxy quantity $\tw$.
\begin{lemma}[Characterization of the tangent cone] \skipline
	\label{lemma:tangent-cone:description}
	Suppose that $\tu \in F$ is a feasible point.
	Then for every function $\tw$ satisfying \cref{assumption:abstract-w-properties}
	\begin{equation*}
		\tangentCone{F}{\tu}
		=
		\setDef[big]{\td \in \R^\nvertices}{\norm{\abs{\sumOfSquaredDifferences{\tu}} + \abs{\sumOfSquaredDifferences{\td}}}_{0,\nu} \le K}
		.
	\end{equation*}
\end{lemma}
\begin{proof}
	We begin by showing the inclusion~\enquote{$\subseteq$}.
	Let $\tu \in F$ and $\td \in \tangentCone{F}{\tu}$.
	By definition, there exist $\sequence{\tau}{k} \searrow 0$, $\sequence[big](){\tu}{k} \subseteq F$ with $\sequence{\tu}{k} \to \tu$ and $\td = \lim_{k \to \infty} (\sequence{\tu}{k} - \tu)/\sequence{\tau}{k}$.
	As $\sumOfSquaredDifferences{\tu}$ depends continuously on~$\tu$, we obtain that $\sumOfSquaredDifferences{\sequence{\tu}{k}} \to \sumOfSquaredDifferences{\tu}$ where $\norm{\sumOfSquaredDifferences{\sequence{\tu}{k}}}_{0,\nu} \le K$ and $\norm{\sumOfSquaredDifferences{\tu}}_{0,\nu} \le K$.
	By \cref{lemma:convergent-nodal-values-detect-non-constant-cells}, there exist $N_1, N_2 \in \N$ such that
	\begin{equation*}
		\begin{alignedat}{3}
			\nonconstantsimplices{\tu}
			&\subseteq
			\nonconstantsimplices{\sequence{\tu}{k}}
			\quad
			&&\text{for all }
			k \ge N_1
			,
			\\
			\nonconstantsimplices{\td}
			&\subseteq
			\nonconstantsimplices[Big]{\frac{\sequence{\tu}{k} - \tu}{\sequence{\tau}{k}}}
			=
			\nonconstantsimplices{\sequence{\tu}{k} - \tu}
			\quad
			&&\text{for all }
			k \ge N_2
			.
		\end{alignedat}
	\end{equation*}
	Additionally, with \cref{Lem:supp}, we obtain
	\begin{equation*}
		\nonconstantsimplices{\sequence{\tu}{k} - \tu}
		\subseteq
		\nonconstantsimplices{\sequence{\tu}{k}}
		\cup
		\nonconstantsimplices{- \tu}
		=
		\nonconstantsimplices{\sequence{\tu}{k}}
		\cup
		\nonconstantsimplices{\tu}
		=
		\nonconstantsimplices{\sequence{\tu}{k}}
	\end{equation*}
	for all $k \ge N_1$.

	Altogether, we have for all $k \ge N \coloneqq \max \set{N_1, N_2}$
	\begin{equation*}
		\nonconstantsimplices{\tu}
		\cup
		\nonconstantsimplices{\td}
		\;
		\subseteq
		\;
		\nonconstantsimplices{\sequence{\tu}{k}}
		\cup
		\nonconstantsimplices{\sequence{\tu}{k} - \tu}
		\;
		\subseteq
		\;
		\nonconstantsimplices{\sequence{\tu}{k}}
	\end{equation*}
	and hence
	\begin{equation*}
		\norm{\abs{\sumOfSquaredDifferences{\tu}} + \abs{\sumOfSquaredDifferences{\td}}}_{0,\nu}
		\le
		\norm{\sumOfSquaredDifferences{\sequence{\tu}{k}}}_{0,\nu}
		\le
		K
		.
	\end{equation*}

	We now proceed with the reverse inclusion~\enquote{$\supseteq$}.
	Suppose that $\tu \in F$ and $\td \in \R^\nvertices$ satisfies $\norm{\abs{\sumOfSquaredDifferences{\tu}} + \abs{\sumOfSquaredDifferences{\td}}}_{0,\nu} \le K$.
	We define $\sequence{\tau}{k} \coloneqq 1/k$ and $\sequence{\tu}{k} \coloneqq \tu + \sequence{\tau}{k} \td$ for all $k \in \N$.
	By \cref{Lem:supp}, we then have for all $k \in \N$
	\begin{equation*}
		\nonconstantsimplices{\sequence{\tu}{k}}
		=
		\nonconstantsimplices{\tu + \sequence{\tau}{k} \td}
		\subseteq
		\nonconstantsimplices{\tu}
		\cup
		\nonconstantsimplices{\sequence{\tau}{k} \td}
		=
		\nonconstantsimplices{\tu}
		\cup
		\nonconstantsimplices{\td}
	\end{equation*}
	and therefore
	\begin{equation*}
		\norm{\sumOfSquaredDifferences{\sequence{\tu}{k}}}_{0,\nu}
		\le
		\norm{\abs{\sumOfSquaredDifferences{\tu}} + \abs{\sumOfSquaredDifferences{\td}}}_{0,\nu}
		\le
		K
		.
	\end{equation*}
	This shows $\td \in \tangentCone{F}{\tu}$.
\end{proof}

We now make the representation of the tangent cone more explicit.
\begin{theorem}
	\label{corollary:tangent-cone:explicit-description}
	Suppose that $\tu \in F$ is a feasible point.
	Then
	\begin{equation}
		\label{eq:tangent-cone:explicit-description}
		\tangentCone{F}{\tu}
		=
		\setDef[Big]
		{\td \in \R^\nvertices}
		{
			\sum_{i\in\nonconstantsimplices{\tu}}\nu_i
			\ + \
			\sum_{\mrep{i \in \constantsimplices{\tu} \cap \nonconstantsimplices{\td}}{}} \nu_i \le K
		}
		.
	\end{equation}
	In particular, for any $\td \in \tangentCone{F}{\tu}$, we also have $-\td \in \tangentCone{F}{\tu}$.
	Therefore, the variational inequality \eqref{eq:B-stationarity} of B-stationarity is equivalent to the variational \emph{equality}
	\begin{equation}
		\label{eq:B-stationarity-with-equality}
		f'(\tu) \, \td
		=
		0
		\quad
		\text{for all }
		\td \in \tangentCone{F}{\tu}
		.
	\end{equation}
\end{theorem}
\begin{proof}
	By \cref{lemma:tangent-cone:description} we know that
	\begin{equation*}
		\tangentCone{F}{\tu}
		=
		\setDef[big]
		{\td \in \R^\nvertices}
		{\norm{\abs{\sumOfSquaredDifferences{\tu}} + \abs{\sumOfSquaredDifferences{\td}}}_{0,\nu} \le K}
		.
	\end{equation*}
	By definition of the weighted $\ell_0$-pseudo-norm, we have
	\begin{align*}
		{\norm{\abs{\sumOfSquaredDifferences{\tu}} + \abs{\sumOfSquaredDifferences{\td}}}_{0,\nu}}
		&
		=
		\sum_{\mrep[l]{i \in \nonconstantsimplices{\tu} \cup \nonconstantsimplices{\td}}{}} \nu_i
		\\
		&
		=
		\sum_{\mrep[c]{i \in \nonconstantsimplices{\tu}}{\hspace*{6mm}}} \nu_i
		\quad
		+
		\quad
		\sum_{\mrep[l]{i \in \constantsimplices{\tu} \cap \nonconstantsimplices{\td}}{}} \nu_i
		.
	\end{align*}
	Due to $\nonconstantsimplices{\td} = \nonconstantsimplices{-\td}$, the proof is complete.
\end{proof}
The previous theorem shows that tangent directions~$\td$ are constrained only in terms of the volume of cells whose $\tu$-vertex values are constant but whose $(\tu + \td)$-vertex values are not.
Notice that a single vertex value of~$\td$ induces changes in the function~$u_h + d_h$ on all incident cells.
In order to connect the cell-based description \eqref{eq:tangent-cone:explicit-description} of the tangent cone with a vertex-based view, we introduce the following definition.

\begin{definition}[Connected set of vertices, covered cells] \skipline
	\label{definition:subsets-of-vertices}
	\begin{enumerate}
		\item
			A subset $V \subseteq \set{1, \ldots, \nvertices}$ of vertices is said to be connected if any two vertices in~$V$ are connected via edges visiting only vertices in~$V$.

		\item
			A cell $S_i$, $i \in \set{1, \ldots, \nsimplices}$, is said to be covered by the subset of vertices $V \subseteq \set{1, \ldots, \nvertices}$ if $\simplextovertices{i} \subseteq V$.
			The set of all cells covered by $V$ is denoted by
			\begin{equation*}
				\coveredsimplices{V}
				\coloneqq
				\setDef[big]{i \in \set{1, \ldots, \nsimplices}}{\simplextovertices{i} \subseteq V}
				.
			\end{equation*}
	\end{enumerate}
\end{definition}

As seen in \eqref{eq:tangent-cone:explicit-description}, tangent directions~$\td \in \tangentCone{F}{\tu}$ are characterized by conditions on the set of cells $\constantsimplices{\tu} \cap \nonconstantsimplices{\td}$.
The main result is now based on using these characterizations to obtain a vertex-based point of view on B-stationarity.
It turns out that the crucial tangent directions~$\td$ are binary vectors, which allow for the representation $\nonconstantsimplices{\td} = \vertextosimplices{V} \setminus \coveredsimplices{V}$ with $V = \supp(\td)$.

\begin{theorem}[Characterization of $B$-stationarity] \skipline
	\label{theorem:B-stationarity:explicit-description}
	Suppose that $\tu \in F$ is a feasible point.
	Then $\tu$ is a B-stationary point of \eqref{eq:general-problem:discrete} if and only if
	\begin{equation}
		\label{eq:B-stationarity:explicit-description}
		\sum_{\ell \in V} \nabla f(\tu)_\ell
		=
		0
		\quad
		\text{for all }
		V \in \cV(\tu)
		,
	\end{equation}
	where
	\begin{equation*}
		\cV(\tu)
		\coloneqq
		\setDef[bigg]{V \subseteq \set{1, \ldots, \nvertices}}{%
			V
			\text{ is connected and }
			\sum_{\mrep{i \in \nonconstantsimplices{\tu}}{}} \nu_i
			\quad
			+
			\quad
			\sum_{\mrep[l]{\begin{subarray}{l} i \in \constantsimplices{\tu} \cap \\ \paren[auto](){\vertextosimplices{V} \setminus \coveredsimplices{V}} \end{subarray}}{}}
			\nu_i
			\le
			K
		}
	\end{equation*}
	contains all subsets of nodes whose neighboring but non-covered cells that overlap the cells on which $\tu$ is constant do not add up to more volume than there is slack in the gradient support constraint.
\end{theorem}
\begin{proof}
	Suppose that $\tu \in F$ is feasible and B-stationary.
	We first show that for any $V \in \cV(\tu)$,
	\begin{equation*}
		\td
		\coloneqq
		\sum_{\ell \in V} \te_\ell
		\in
		\tangentCone{F}{\tu}
	\end{equation*}
	holds, where $\te_\ell$ is the $\ell$-th standard basis vector in $\R^\nvertices$.
	For such $\td$, a cell $S_i$ belongs to $\nonconstantsimplices{\td}$ if and only if $S_i$ intersects but is not covered by~$V$, \ie,
	\begin{equation*}
		\constantsimplices{\tu} \cap \nonconstantsimplices{\td}
		=
		\constantsimplices{\tu} \cap \paren[big](){\vertextosimplices{V} \setminus \coveredsimplices{V}}
		.
	\end{equation*}
	Due to $V \in \cV(\tu)$ and the description \eqref{eq:tangent-cone:explicit-description} of the tangent cone, we obtain $\td \in \tangentCone{F}{\tu}$.
	$B$-stationarity \eqref{eq:B-stationarity-with-equality} gives
	\begin{equation*}
		\sum_{\ell \in V} \nabla f(\tu)_\ell
		=
		f'(\tu) \, \td
		=
		0
		,
	\end{equation*}
	and consequently \eqref{eq:B-stationarity:explicit-description} holds.

	Conversely, assume that $\tu \in F$ satisfies condition \eqref{eq:B-stationarity:explicit-description}.
	We have to show $f'(\tu) \, \td \ge 0$ for all $\td \in \tangentCone{F}{\tu}$.

	Let $\td \in \tangentCone{F}{\tu}$ be arbitrary but fixed.
	By \cref{corollary:tangent-cone:explicit-description}, we know that
	\begin{equation*}
		\sum_{i \in \nonconstantsimplices{\tu}} \nu_i
		\quad
		+
		\quad
		\sum_{\mrep{i \in \constantsimplices{\tu} \cap \nonconstantsimplices{\td}}{}} \nu_i
		\le
		K
		.
	\end{equation*}
	Let us decompose $\td$ into a linear combination of binary vectors~$\sequence{\td}{k} \in \set{0,1}^\nvertices$ whose support~$\sequence{V}{k}$ belongs to~$\cV(\tu)$.
	To this end, select a vertex from the support of~$\td$, and extend it to the maximal connected set~$\sequence{V}{1}$ of vertices of equal $\td$-value, referred to as $\sequence{a}{1}$.
	Repeat this process with the remaining vertices that are not already in such a maximally connected set until the support of~$\td$ is exhausted.
	This process yields a unique decomposition
	\begin{equation*}
		\td
		=
		\sum_{k=1}^{n_\td} \sequence{a}{k} \sequence{\td}{k}
		=
		\sum_{k=1}^{n_\td} \sequence{a}{k} \sum_{\ell \in \sequence{V}{k}} \te_\ell
	\end{equation*}
	with $n_\td \in \N_0$ and $\sequence{a}{k} \neq 0$ for all $k = 1, \ldots, n_\td$.
	Indeed, we obtain $\sequence{V}{k} \in \cV(\tu)$ because
	\begin{align*}
		\sum_{\mrep[r]{i \in \constantsimplices{\tu} \cap \paren[auto](){\vertextosimplices{\sequence{V}{k}} \setminus \coveredsimplices{\sequence{V}{k}}}}{\hspace*{6mm}}}
		\nu_i
		&
		=
		\sum_{\mrep[l]{i \in \constantsimplices{\tu} \cap \nonconstantsimplices{\sequence{\td}{k}}}{\hspace*{10mm}}} \nu_i
		\\
		&
		\le
		\sum_{\mrep[l]{i \in \constantsimplices{\tu} \cap \nonconstantsimplices{\td}}{\hspace*{10mm}}}
		\nu_i
		\\
		&
		\le
		K
		-
		\sum_{\mrep[l]{i \in \nonconstantsimplices{\tu}}{\hspace*{10mm}}} \nu_i
		.
	\end{align*}
	In this estimate, the first equality is due to the equality $\vertextosimplices{\sequence{V}{k}} \setminus \coveredsimplices{\sequence{V}{k}}=\nonconstantsimplices{\sequence{\td}{k}}$, which is true because $\sequence{V}{k}$ is defined as the support of $\sequence{\td}{k}$.
	The first inequality is due to $\nonconstantsimplices{\sequence{\td}{k}} \subseteq \nonconstantsimplices{\td}$, which is true because any $i\in \nonconstantsimplices{\sequence{\td}{k}} \setminus \nonconstantsimplices{\td}$ would correspond to at least one vertex that could have been added to $\sequence{V}{k}$ to obtain a larger connected set of vertices to define $\sequence{\td}{k}$ on and therefore a contradiction to the maximally constructed $\sequence{V}{k}$.
	Finally, the second inequality follows from \eqref{eq:tangent-cone:explicit-description}.

	Accordingly, we see that
	\begin{equation*}
		f'(\tu) \, \td
		=
		\sum_{k=1}^{n_\td} \sequence{a}{k} f'(\tu) \sum_{\ell \in \sequence{V}{k}} \te_\ell
		=
		\sum_{k=1}^{n_d} \sequence{a}{k} \sum_{\ell \in \sequence{V}{k}} \nabla f(\tu)_\ell
		=
		0
		,
	\end{equation*}
	where the last equality follows from \eqref{eq:B-stationarity:explicit-description}.
\end{proof}

\Cref{figure:subsets-of-vertices} illustrates the essential quantities in the formulation of \cref{theorem:B-stationarity:explicit-description} and the two following corollaries that consider the special cases of \eqref{eq:B-stationarity:explicit-description} where $V$ consists of a single vertex that is not vertex of any of the simplices where $u_h$ is constant (\cref{figure:subsets-of-vertices:c}), and the case where $V$ covers all cells in $\constantsimplices{\tu}$ (\cref{figure:subsets-of-vertices:d}), respectively.

\begin{figure}[htp]
	\centering
	\subfloat[$u_h$ constant.]{%
		\begin{tikzpicture}
			\foreach \i in {0, ..., 3}{
				\foreach \j in {0, ..., 4}{
					\draw[black] (\i,\j) -- (\i+1,\j);
				}
			}

			\foreach \i in {0, ..., 3}{
				\foreach \j in {0, ..., 3}{
					\draw[black] (\i,\j) -- (\i,\j+1);
					\draw[black] (\i,\j) -- (\i+1,\j+1);
				}
			}

			\draw[black] (4,0) -- (4,4);

			\fill[color = red, opacity = 0.3] (0,0) -- (0,2) -- (1,2) -- (2,3) -- (3,3) -- (3,2) -- (2,1) -- (1,0) -- cycle;

			\path (0,0) circle(1.5pt);
		\end{tikzpicture}
		\label{figure:subsets-of-vertices:a}
	}
	\quad
	\subfloat[General situation of \cref{theorem:B-stationarity:explicit-description}.]{%
		\begin{tikzpicture}
			\foreach \i in {0, ..., 3}{
				\foreach \j in {0, ..., 4}{
					\draw[black] (\i,\j) -- (\i+1,\j);
				}
			}

			\foreach \i in {0, ..., 3}{
				\foreach \j in {0, ..., 3}{
					\draw[black] (\i,\j) -- (\i,\j+1);
					\draw[black] (\i,\j) -- (\i+1,\j+1);
				}
			}

			\draw[black] (4,0) -- (4,4);

			\fill[color = red, opacity = 0.3] (0,0) -- (0,2) -- (1,2) -- (2,3) -- (3,3) -- (3,2) -- (2,1) -- (1,0) -- cycle;

			\fill[color = blue, opacity = 0.3] (0,0) -- (0,2) -- (2,4) -- (3,4) -- (3,2) -- (2,1) -- (1,1) -- cycle;

			\draw[blue, fill=blue] (0,1) circle(1.5pt);
			\draw[blue, fill=blue] (1,2) circle(1.5pt);
			\draw[blue, fill=blue] (2,3) circle(1.5pt);
			\draw[blue, fill=blue] (2,2) circle(1.5pt);
			\draw[blue, fill=blue] (1,2) circle(1.5pt);

			\draw[blue, pattern = north east lines, pattern color = blue] (0,1) -- (2,3) -- (2,2) -- (1,2) -- cycle;

			\path (0,0) circle(1.5pt);
		\end{tikzpicture}
		\label{figure:subsets-of-vertices:b}
	}

	\subfloat[Special situation of \cref{corollary:B-stationarity:outside}.]{%
		\begin{tikzpicture}
			\foreach \i in {0, ..., 3}{
				\foreach \j in {0, ..., 4}{
					\draw[black] (\i,\j) -- (\i+1,\j);
				}
			}

			\foreach \i in {0, ..., 3}{
				\foreach \j in {0, ..., 3}{
					\draw[black] (\i,\j) -- (\i,\j+1);
					\draw[black] (\i,\j) -- (\i+1,\j+1);
				}
			}

			\draw[black] (4,0) -- (4,4);

			\fill[color = red, opacity = 0.3] (0,0) -- (0,2) -- (1,2) -- (2,3) -- (3,3) -- (3,2) -- (2,1) -- (1,0) -- cycle;

			\fill[color = blue, opacity = 0.3] (2,0) -- (2,1) -- (3,2) -- (4,2) -- (4,1) -- (3,0) -- cycle;

			\draw[blue, fill=blue] (3,1) circle(1.5pt);

			\path (0,0) circle(1.5pt);
		\end{tikzpicture}
		\label{figure:subsets-of-vertices:c}
	}
	\quad
	\subfloat[Special situation of \cref{corollary:B-stationarity:covered}.]{%
		\begin{tikzpicture}
			\fill[color = red, opacity = 0.3] (0,0) -- (0,2) -- (1,2) -- (2,3) -- (3,3) -- (3,2) -- (2,1) -- (1,0) -- cycle;

			\fill[color = blue, opacity = 0.3] (0,0) -- (0,3) -- (1,4) -- (2,4) -- (4,4) -- (4,2) -- (2,0) -- cycle;

			\fill[blue, pattern = north east lines, pattern color = blue] (0,0) -- (0,2) -- (1,3) -- (2,3) -- (3,3) -- (3,2) -- (2,1) -- (1,0) -- cycle;

			\draw[black] (4,0) -- (4,4);

			\foreach \i in {0, ..., 3}{
				\foreach \j in {0, ..., 4}{
					\draw[black] (\i,\j) -- (\i+1,\j);
				}
			}

			\foreach \i in {0, ..., 3}{
				\foreach \j in {0, ..., 3}{
					\draw[black] (\i,\j) -- (\i,\j+1);
					\draw[black] (\i,\j) -- (\i+1,\j+1);
				}
			}
			\draw[blue, fill=blue] (0,0) circle(1.5pt);
			\draw[blue, fill=blue] (1,0) circle(1.5pt);
			\draw[blue, fill=blue] (0,1) circle(1.5pt);
			\draw[blue, fill=blue] (1,1) circle(1.5pt);
			\draw[blue, fill=blue] (2,1) circle(1.5pt);
			\draw[blue, fill=blue] (0,2) circle(1.5pt);
			\draw[blue, fill=blue] (1,2) circle(1.5pt);
			\draw[blue, fill=blue] (2,2) circle(1.5pt);
			\draw[blue, fill=blue] (3,2) circle(1.5pt);
			\draw[blue, fill=blue] (2,3) circle(1.5pt);
			\draw[blue, fill=blue] (3,3) circle(1.5pt);
			\draw[blue, fill=blue] (1,3) circle(1.5pt);

			\draw[color = blue] (0,0) -- (0,2) -- (1,2) -- (2,3) -- (3,3) -- (3,2) -- (2,1) -- (1,0) -- cycle;
		\end{tikzpicture}
		\label{figure:subsets-of-vertices:d}
	}
	\caption[]{%
		Consider a feasible function $u_h \in U_h$ on the triangulated unit square that is constant exactly on the simplices marked in red in \cref{figure:subsets-of-vertices:a} with the corresponding index set $\constantsimplices{\tu}$.
		The quantity $\norm{\sumOfSquaredDifferences{\tu}}_{0,\nu}$ is the area of the complement (white cells).

		\Cref{figure:subsets-of-vertices:b} shows a connected set of four vertices $\sequence{V}{1}$ (blue circles).
		The simplices corresponding to $\vertextosimplices{\sequence{V}{1}}$ are overlayed in blue.
		The set $\sequence{V}{1}$ covers one cell $\coveredsimplices{\sequence{V}{1}}$ hatched in blue.
		The cells in the intersection $\constantsimplices{\tu} \cap \paren[auto](){\vertextosimplices{\sequence{V}{1}} \setminus \coveredsimplices{\sequence{V}{1}}}$ are the ones in purple, without the hatched cell.
		The set $\sequence{V}{1}$ belongs to $\cV(\tu)$ if and only if the combined area of this intersection (purple, non-hatched) and the previously white cells does not exceed~$K$.
		Note that $\sequence{V}{1}$ is not maximally connected, in contrast to the proof of \cref{theorem:B-stationarity:explicit-description}.

		\Cref{figure:subsets-of-vertices:c} shows a vertex set $\sequence{V}{2}$ consisting of a single vertex (blue circle).
		The simplices corresponding to $\vertextosimplices{\sequence{V}{2}}$ are blue.
		No cells are covered by $\sequence{V}{2}$, hence $\coveredsimplices{\sequence{V}{2}} = \emptyset$.
		Because the colored regions do not overlap, there are no purple, non-hatched cells, \ie we have $\constantsimplices{\tu} \cap \paren[auto](){\vertextosimplices{\sequence{V}{2}} \setminus \coveredsimplices{\sequence{V}{2}}} = \emptyset$, so $\sequence{V}{2} \in \cV(\tu)$.

		\Cref{figure:subsets-of-vertices:d} shows a vertex set $\sequence{V}{3}$ (blue circles).
		The simplices corresponding to $\vertextosimplices{\sequence{V}{3}}$ are overlayed blue.
		All cells corresponding to $\constantsimplices{\tu}$ are in $\coveredsimplices{\sequence{V}{3}}$ (hatched).
		Again, there are no purple, non-hatched cells, thus $\sequence{V}{3} \in \cV(\tu)$.
	}
	\label{figure:subsets-of-vertices}
\end{figure}

\begin{corollary}
	\label{corollary:B-stationarity:outside}
	Suppose that $\tu \in F$ is feasible and B-stationary.
	Then
	\begin{equation*}
		\nabla f(\tu)_\ell
		=
		0
		\quad
		\text{for all }
		\ell
		\in
		\set{1,\dots,\nvertices} \setminus \simplextovertices{\constantsimplices{\tu}}
		.
	\end{equation*}
\end{corollary}
\begin{proof}
	Let $V \coloneqq \set{\ell}$ for an index $\ell \in \set{1, \dots, \nvertices} \setminus \simplextovertices{\constantsimplices{\tu}}$.
	Then $\vertextosimplices{V} \subseteq \nonconstantsimplices{\tu}$ holds and therefore
	\begin{equation*}
		\constantsimplices{\tu}
		\cap
		\paren[auto](){\vertextosimplices{V} \setminus \coveredsimplices{V}}
		\;
		\subseteq
		\;
		\constantsimplices{\tu}
		\cap
		\vertextosimplices{V}
		\;
		\subseteq
		\;
		\constantsimplices{\tu}
		\cap
		\nonconstantsimplices{\tu}
		=
		\emptyset
		,
	\end{equation*}
	so $V \in \cV(\tu)$, and the result immediately follows from \cref{theorem:B-stationarity:explicit-description}.
\end{proof}

\begin{corollary}
	\label{corollary:B-stationarity:covered}
	Suppose that $\tu \in F$ is feasible and B-stationary.
	Consider a connected $V \subseteq \set{1, \ldots, \nvertices}$ such that $\constantsimplices{\tu} \subseteq \coveredsimplices{V}$.
	Then
	\begin{equation*}
		\sum_{\ell \in V} \nabla f(\tu)_\ell
		=
		0
		.
	\end{equation*}
\end{corollary}
\begin{proof}
	The assumptions on~$V$ imply that
	\begin{equation*}
		\constantsimplices{\tu}
		\cap
		\paren[auto](){\vertextosimplices{V} \setminus \coveredsimplices{V}}
		=
		\emptyset
		,
	\end{equation*}
	so $V \in \cV(\tu)$ holds, and the result immediately follows from \cref{theorem:B-stationarity:explicit-description}.
\end{proof}
The statement of \cref{corollary:B-stationarity:covered} even holds when $V$ is not assumed to be connected, however this is not a straight-forward corollary of \cref{theorem:B-stationarity:explicit-description} but requires the use of \cref{corollary:tangent-cone:explicit-description}.

In case the cardinality constraint is active at~$\tu$ in a relaxed sense that takes the discrete mesh structure into account, we can further refine the characterization of B-stationarity in \cref{theorem:B-stationarity:explicit-description}.
We refer to this situation as exact feasibility, and it does not allow for any additional non-constant behavior.
\begin{definition}[Exact feasibility] \skipline
	Let $K \in \interval[][]{0}{\mu(\Omega)}$.
	We say that $\tu \in F$ is exactly feasible with respect to the constraint $\norm{\sumOfSquaredDifferences{\tu}}_{0,\nu} \le K$ if
	\begin{equation*}
		\norm{\sumOfSquaredDifferences{\tu}}_{0,\nu} + \nu_i
		>
		K
	\end{equation*}
	holds for all $i \in \constantsimplices{\tu}$.
\end{definition}

\begin{corollary}[Characterization of $B$-stationarity in the exactly feasible case]
	Let $\tu \in F$ be exactly feasible.
	Then $\tu$ is a B-stationary point of \eqref{eq:general-problem:discrete} if and only if
	\begin{equation}
		\label{eq:B-stationarity:explicit-description:exact-feasibility}
		\sum_{\ell \in V} \nabla f(\tu)_\ell
		=
		0
	\end{equation}
	for all connected $V \subseteq \set{1, \dots, \nvertices}$ such that $\constantsimplices{\tu} \cap \paren[auto](){\vertextosimplices{V} \setminus \coveredsimplices{V}} = \emptyset$.
\end{corollary}
\begin{proof}
	The claim immediately follows from \cref{theorem:B-stationarity:explicit-description} because for exactly feasible~$\tu$, the second sum in the definition of $\cV(\tu)$ has to have an empty index set.
\end{proof}

\subsection{Algorithmic Treatment of the Discrete DC Reformulation}
\label{subsection:DC-reformulation:DC-approach}
In this section, we discuss the application of the DC~algorithm to solve the discrete penalized problem \eqref{eq:general-problem:discrete:penalized}, \ie, the problem
\begin{equation}
	\label{eq:general-problem:discrete:penalized:restated-fully}
	\text{Minimize}
	\quad
	f_\rho(\tu)
	\coloneqq
	\paren[auto](){f(\tu)
		+
	\rho \, \norm{\sumOfSquaredDifferences{\tu}}_{1, \nu}}
	-
	\rho\largestKnorm{\sumOfSquaredDifferences{\tu}}{K}[\nu]
	\quad
	\text{where }
	\tu \in \R^\nvertices
	,
\end{equation}
and state the algorithm's convergence properties, given the following assumption, which makes \eqref{eq:general-problem:discrete:penalized:restated-fully} a DC-type problem, because $\tu \mapsto \norm{\sumOfSquaredDifferences{\tu}}_{1,\nu}$ and $\tu \mapsto \largestKnorm{\sumOfSquaredDifferences{\tu}}{K}[\nu]$ are convex under \cref{assumption:abstract-w-properties}.
\begin{assumption}[Additional standing assumptions for the discrete setting] \skipline
	\label{assumption:standing-assumptions:finite-dimensions:differentiability-and-convexity}
	In addition to \cref{assumption:standing-assumptions:finite-dimensions:differentiability}, the objective function~$f$ is assumed to be convex.
\end{assumption}
Note that $f$ is automatically bounded from below by Weierstraß' theorem; see, \eg, \cite[Theorem~1.14]{DharaDutta:2011:1}.

Analogously to \cref{theorem:general-problem:penalized:convergence} in the infinite-dimensional setting we have the following result on the finite-dimensional penalty method:
\begin{theorem}[Convergence of global minimizers] \skipline
	\label{theorem:general-problem:discrete:penalized:convergence}
	Suppose that $\sequence{\rho}{k} > 0$ is a sequence such that $\lim_{k \to \infty} \sequence{\rho}{k} = \infty$.
	Let $\sequence{\tu}{k}$ be a global solution of the penalized problem \eqref{eq:general-problem:discrete:penalized:restated-fully} with $\rho = \sequence{\rho}{k}$ for all $k \in \N$.
	Then the sequence $\sequence[big](){\tu}{k}$ is bounded in $\R^\nvertices$, satisfies $\lim_{k \to \infty} \phi(\sequence{\tu}{k}) = 0$, and every accumulation point $\overline{\tu}$ is a global solution of \eqref{eq:general-problem:discrete}.
\end{theorem}

In each step of the DC~algorithm (see \cref{algorithm:DC-algorithm} below), we choose a subgradient $\sequence{\ts}{k} \in \partial \largestKnorm{\sumOfSquaredDifferences{\cdot}}{K}[\nu] (\sequence{\tu}{k})$ and determine the next iterate $\sequence{\tu}{k+1}$ by solving the convex problem
\begin{equation}
	\label{eq:general-problem:discrete:DC-subproblem}
	\text{Minimize}
	\quad
	f(\tu)
	+
	\rho \, \norm{\sumOfSquaredDifferences{\tu}}_{1,\nu}
	-
	\rho \, \sequence[big](){\ts}{k}^\transp \tu
	\quad
	\text{where }
	\tu \in \R^\nvertices
	.
\end{equation}

In this section, we will work out the specific steps of this approach as well as some results on the behavior.
For this specification, we first need to fix some more structure on the quantity $\tw(\cdot)$, which indicates which simplices a finite element function is (non-)constant on and so far has only been subject to \cref{assumption:abstract-w-properties}.
\begin{assumption}
	\label{assumption:w-matrix-structure}
	There exist matrices $\tD_1 \in \R^{\nedges \times \nvertices}$ and $\tD_2 \in \R^{\nsimplices \times \nedges}$, where $\tD_2$ has nonnegative entries, such that
	\begin{equation}
		\label{eq:sum-of-squared-differences}
		\sumOfSquaredDifferences{\tu}
		=
		\tD_2 \paren[big][]{(\tD_1 \tu) \odot (\tD_1 \tu)}
		\quad
		\text{for all }
		\tu \in \R^\nvertices
		,
	\end{equation}
	where $\odot$ denotes componentwise multiplication.
\end{assumption}
Recall that $\nedges \in \N$ is the number of edges.
Additionally, the quantity $(\tD_1 \tu)_j$ is the difference of the vertex values on edge~$j$, and $\tD_2$ encodes which edges belong to which cell, so that $\sumOfSquaredDifferences{\tu}_i$ is the sum of the squared differences of the vertex values on~$S_i$; see \eqref{eq:D1:finite-elements} and \eqref{eq:D2:finite-elements} below for the specifics.
More generally though, $\nedges$ should be thought of as the number of differences considered, and $\tD_2$ groups them by the entries of $\tw$.
Not making this fully specific at this time will allow us to comment on the use of the DC Algorithm in an image reconstruction context later.
Additionally, note that, since $\tD_2$ has nonnegative entries, \cref{assumption:w-matrix-structure} implies convexity of $\tw$, which is in line with \cref{assumption:abstract-w-properties}.

\begin{remark}
	\label{remark:not-squaring-the-differences-of-vertex-values}
	Following \cite[Section~6.1]{DittrichWachsmuth:2025:1}, one could also consider $\tD_2 \abs{\tD_1 \tu} \in \R^\nsimplices$ (with the componentwise absolute value) in place of \eqref{eq:sum-of-squared-differences}.
	For that choice, however, the rank deficiency of $\tD_1$ for standard finite element meshes leads to a singular matrix in the semismooth Newton method one would like to apply when solving the resulting DC subproblems, \cf \cite[Section~6.4]{DittrichWachsmuth:2025:1}.
	Note that the matrix $\tD_1$ would have full rank if we would only consider the rows pertaining to edges that form a minimum spanning tree (\cite[Section~1.5]{Diestel:2025:1}).
	Unfortunately, when $\tD_1$ is generated from such a minimum spanning tree, then, in general, $\paren[big](){\tD_2 \abs{\tD_1 \tu}}_i = 0$ is no longer indicative of $u_h$ being constant on~$S_i$.
	An additional benefit of using the squared differences \eqref{eq:sum-of-squared-differences} is that the resulting DC subproblems \eqref{eq:general-problem:discrete:DC-subproblem} are smooth, see \cref{lemma:norm-of-sum-of-squared-differences-via-matrices}.
	For instance, when $f$ is a quadratic polynomial, we simply need to solve a linear system instead of employing a semismooth Newton method; see \eqref{eq:DC-subproblem:optimality-conditions}.
	In retrospect, we also tried using a squared formulation in the setting of \cite[Section~6.1]{DittrichWachsmuth:2025:1}.
	This would also render the resulting DC subproblems smooth.
	However, numerical comparisons show that this can lead to significantly slower convergence or convergence to local minimizers in that setting.
\end{remark}

First, we need to determine a way to find a subgradient $\sequence{\ts}{k}$.
\begin{lemma}
	\label{lemma:norm-of-sum-of-squared-differences-via-matrices}
	Let $\tu \in \R^\nvertices$, $K \in \interval[][]{0}{\norm{\nu}_1}$ and $I \subseteq \set{1, \ldots, \nsimplices}$ be such that the maximum in $\largestKnorm{\sumOfSquaredDifferences{\tu}}{K}[\nu]$ is attained.
	Then we have that
	\begin{equation*}
		\norm{\sumOfSquaredDifferences{\tu}}_{1,\nu}
		=
		\tu^\transp \tM \tu
		,
		\qquad
		\largestKnorm{\sumOfSquaredDifferences{\tu}}{K}[\nu]
		=
		\tu^\transp \tM_I \tu
		,
	\end{equation*}
	for the symmetric and positive semidefinite matrices $\tM \coloneqq \tD_1^\transp \diag \paren[auto](){\tD_2^\transp \nu} \tD_1$ and $\tM_I \coloneqq \tD_1^\transp \diag\paren[auto](){\tD_2^\transp \indicatorFunction{I} \nu} \, \tD_1$ in $\R^{\nvertices \times \nvertices}$.
	Here $\indicatorFunction{I} \nu \in \R^\nsimplices$ is defined componentwise via $(\indicatorFunction{I} \nu)_i \coloneqq \indicatorFunction{I}(i) \, \nu_i$ for the indicator function $\indicatorFunction{I}$ of $I$ with values in $\set{0,1}$.
\end{lemma}
\begin{proof}
	Since $\tD_2$ has nonnegative entries by \cref{assumption:w-matrix-structure}, we have $\sumOfSquaredDifferences{\tu} \ge 0$ for all $\tu \in \R^\nvertices$, so the absolute values in the weighted norms can be omitted.
	To see the right-hand identity, note that
	\begin{equation*}
		\begin{aligned}
			\largestKnorm{\sumOfSquaredDifferences{\tu}}{K}[\nu]
			&
			=
			\sum_{i \in I} \nu_i (\tD_2 \paren[big][]{(\tD_1 \tu) \odot (\tD_1 \tu)})_i
			\\
			&
			=
			\sum_{i=1}^\nsimplices \indicatorFunction{I}(i) \, \nu_i \sum_{j=1}^\nedges (\tD_2)_{i,j} (\tD_1 \tu)_j^2
			\\
			&
			=
			\sum_{j=1}^\nedges (\tD_1 \tu)_j^2 \sum_{i = 1}^\nsimplices (\tD_2)_{i,j} \indicatorFunction{I}(i) \, \nu_i
			\\
			&
			=
			\sum_{j=1}^\nedges (\tD_1 \tu)_j^2 \, (\tD_2^\transp \indicatorFunction{I} \nu)_j
			\\
			&
			=
			\tu^\transp \tM_I \tu
			.
		\end{aligned}
	\end{equation*}
	By definition, $\tM_I$ is symmetric.
	Moreover, its positive semidefiniteness follows from $\largestKnorm{\sumOfSquaredDifferences{\tu}}{K}[\nu] \ge 0$ for all $\tu \in \R^\nvertices$, showing the right-hand identity for arbitrary $K \in \interval[][]{0}{\norm{\nu}_1}$ and corresponding~$I$.
	The left-hand claim for $\norm{\sumOfSquaredDifferences{\tu}}_{1,\nu}$ and $\tM$ follows from the previous result by setting $K = \norm{\nu}_1$ with corresponding $I = \set{1, \ldots, \nsimplices}$.
\end{proof}

With the previous result, we can provide a full characterization of the subdifferential in question.
\begin{theorem}[\protect{Characterization of the subdifferential $\partial \largestKnorm{\sumOfSquaredDifferences{\cdot}}{K}[\nu]$}] \skipline
	\label{theorem:subdifferential-of-the-largest-K-norm-of-the-sum-of-squared-differences}
	Let $\tu \in \R^\nvertices$.
	Then we have
	\begin{multline*}
		\partial \largestKnorm{\sumOfSquaredDifferences{\cdot}}{K}[\nu](\tu)
		\\
		=
		\conv
		\setDef[Big]
		{2 \, \tM_I \tu}
		{%
			I \subseteq \set{1, \ldots, \nsimplices}
			,
			\;
			\sum_{i \in I} \nu_i \le K
			,
			\;
			\sum_{i \in I} \nu_i \, \sumOfSquaredDifferences{\tu}_i = \largestKnorm{\sumOfSquaredDifferences{\tu}}{K}[\nu]
		}
		.
	\end{multline*}
	Moreover,
	\begin{equation}
		\label{eq:subdifferential-of-the-largest-K-norm-of-the-sum-of-squared-differences:choice}
		2 \, \tD_1^\transp \paren[auto](){\tD_2^\transp \tr \odot \tD_1 \tu}
		\in
		\partial \largestKnorm{\sumOfSquaredDifferences{\cdot}}{K}[\nu](\tu)
	\end{equation}
	holds for any $\tr \in \partial \largestKnorm{\cdot}{K}[\nu](\sumOfSquaredDifferences{\tu})$, where $\odot$ still denotes componentwise multiplication.
\end{theorem}
\begin{proof}
	By \cref{lemma:norm-of-sum-of-squared-differences-via-matrices}, we have
	\begin{equation*}
		\largestKnorm{\sumOfSquaredDifferences{\tu}}{K}[\nu]
		=
		\max \setDef[Big]{\tu^\transp \tM_I \tu}{I \subseteq \set{1, \ldots, \nsimplices} \text{ \st\ } \smash{\sum_{i \in I}} \nu_i \le K}
	\end{equation*}
	with $\tM_I$ symmetric and positive semidefinite.
	Hence, $\largestKnorm{\sumOfSquaredDifferences{\tu}}{K}[\nu]$ is the maximum of a finite number of real, convex and continuously differentiable functions, which by \cite[Corollary~4.4.4]{HiriartUrrutyLemarechal:1993:1} shows the first claim.

	The second claim also follows from the chain rule for the convex subdifferential, see \cite[Theorem~13.23]{ClasonValkonen:2025:1}, applied to $\largestKnorm{\cdot}{K}[\nu]$ and $\sumOfSquaredDifferences{\cdot}$, where by definition
	\begin{align*}
		\nabla \sumOfSquaredDifferences{\tu}^\transp \tr
		&
		=
		\paren[auto](){\nabla \tD_2 \paren[big][]{(\tD_1 \tu) \odot (\tD_1 \tu)}}^\transp \tr
		\\
		&
		=
		\paren[auto](){2 \, \tD_2 \paren[big][]{\diag(\tD_1 \tu) \, \tD_1}}^\transp \tr
		\\
		&
		=
		2 \, \tD_1^\transp \diag(\tD_1 \tu) \, \tD_2^\transp \tr
		\\
		&
		=
		2 \, \tD_1^\transp \paren[auto](){\tD_2^\transp \tr \odot  (\tD_1 \tu) }
	\end{align*}
	holds for every $\tr \in \partial \largestKnorm{\cdot}{K}[\nu](\sumOfSquaredDifferences{\tu})$.
\end{proof}

Next, the convex subproblem \eqref{eq:general-problem:discrete:DC-subproblem} of the DC~algorithm has to be solved.
By \cref{lemma:norm-of-sum-of-squared-differences-via-matrices}, problem \eqref{eq:general-problem:discrete:DC-subproblem} is equivalent to
\begin{equation*}
	\text{Minimize}
	\quad
	f(\tu)
	+
	\rho \, \tu^\transp \tM \tu
	-
	\rho \, \sequence[big](){\ts}{k}^\transp \tu
	\quad
	\text{where }
	\tu \in \R^\nvertices
	.
\end{equation*}
Due to convexity, the solutions are fully characterized by the optimality condition
\begin{equation}
	\label{eq:DC-subproblem:optimality-conditions}
	\nabla f(\tu)
	+
	2 \, \rho \, \tM \tu
	-
	\rho \, \sequence{\ts}{k}
	=
	0
	.
\end{equation}
If $f$ is, for instance, a quadratic polynomial $f(\tu) = \frac{1}{2} \tu^\transp A \tu - \tb^\transp \tu$ with $A$ symmetric and positive definite, then the above equation has the unique solution $(A + 2 \, \rho \, \tM)^{-1} (\tb + \rho \, \sequence{\ts}{k})$.

The specific DC~algorithm for solving \eqref{eq:general-problem:discrete:penalized:restated-fully} is hence given by \cref{algorithm:DC-algorithm}.
Note that there is some flexibility in choosing a subgradient $\sequence{\tr}{k} \in \partial \largestKnorm{\cdot}{K}[\nu](\sumOfSquaredDifferences{\sequence{\tu}{k}})$ in \cref{algorithm:DC-algorithm:determine-subgradient}.
For example, one can choose an index set $\sequence{I}{k} \subseteq \set{1, \ldots, \nsimplices}$ such that $\sum_{i \in \sequence{I}{k}} \nu_i \le K$ and $\sum_{i \in \sequence{I}{k}} \nu_i \sumOfSquaredDifferences{\sequence{\tu}{k}}_i = \largestKnorm{\sumOfSquaredDifferences{\sequence{\tu}{k}}}{K}[\nu]$ is satisfied, and set
\begin{equation}
	\label{eq:choice_of_subgradient}
	\sequence{\tr_i}{k}
	\coloneqq
	\indicatorFunction{\sequence{I}{k}} \nu
	=
	\begin{cases}
		\nu_i
		,
		&
		i \in \sequence{I}{k}
		\\
		0
		,
		&
		i \in \set{1, \ldots, \nsimplices} \setminus \sequence{I}{k}
		.
	\end{cases}
\end{equation}
Note that this results in $\sequence{\ts}{k} = 2 \, \tM_{\sequence{I}{k}} \sequence{\tu}{k}$.

\begin{algorithm}[htp]
	\caption{DC algorithm for solving the discrete penalized problem \eqref{eq:general-problem:discrete:penalized:restated-fully}.}
	\label{algorithm:DC-algorithm}
	\begin{algorithmic}[1]
		\Require
		$\sequence{\tu}{0} \in \R^\nvertices$, $\rho >0$

		\Ensure
		Approximate solution $\sequence{\tu}{k}$ to \eqref{eq:general-problem:discrete:penalized:restated-fully}

		\State
		Set $k \gets 0$

		\While{Termination criterion not satisfied by $\sequence{\tu}{k}$}

		\State
		\label{algorithm:DC-algorithm:determine-subgradient}%
		Determine $\sequence{\tr}{k} \in \partial \largestKnorm{\cdot}{K}[\nu](\sumOfSquaredDifferences{\sequence{\tu}{k}})$.

		\State
		\label{algorithm:DC-algorithm:solve-subproblem}%
		Set $\sequence{\ts}{k} \coloneqq 2 \, \tD_1^\transp \paren[auto](){\tD_2^\transp \sequence{\tr}{k} \odot \tD_1 \sequence{\tu}{k}}$ and determine $\sequence{\tu}{k+1}$ by solving
		\begin{equation*}
			\nabla f(\tu) + 2 \, \rho \, \tM \tu
			=
			\rho \, \sequence{\ts}{k}
			.
		\end{equation*}

		\State
		Set $k \leftarrow k+1$

		\EndWhile
	\end{algorithmic}
\end{algorithm}

We move on to the DC algorithm's well-known convergence properties.
Recall \cref{definition:critical-point} of critical points, which carries over to the case where $U = \R^n$.
We can characterize critical points as follows.
\begin{lemma}[Characterization of critical points of the penalized problem \eqref{eq:general-problem:discrete:penalized:restated-fully}]
	\label{lemma:general-problem:discrete:penalized:multiplier}
	A point $\tu$ is a critical point of $f_\rho$ \wrt the natural DC decomposition in \eqref{eq:general-problem:discrete:penalized:restated-fully} if and only if
	\begin{equation*}
		\nabla f(\tu)
		+
		2 \, \rho \, \tM \tu
		\in
		\rho \, \partial \largestKnorm{\sumOfSquaredDifferences{\cdot}}{K}[\nu](\tu)
		,
	\end{equation*}
	\ie, if and only if there exists $\lambda \in 2 \, \rho \, \tM \tu - \rho \, \partial \largestKnorm{\sumOfSquaredDifferences{\cdot}}{K}[\nu](\tu)$
	such that
	\begin{equation}
		\nabla f(\tu) + \lambda
		=
		0
		.
		\label{eq:general-problem:discrete:penalized:multiplier}
	\end{equation}
	Furthermore, given a critical point $\tu$, the following are equivalent:
	\begin{enumerate}
		\item \label[statement]{item:general-problem:discrete:penalized:multiplier:1}
			$\tu$ is strongly critical.

		\item \label[statement]{item:general-problem:discrete:penalized:multiplier:2}
			$\partial \largestKnorm{\sumOfSquaredDifferences{\cdot}}{K}[\nu](\tu)$ is a singleton.
	\end{enumerate}
	In this case, for any index set $I \subseteq \set{1, \dots, \nsimplices}$ that realizes the maximum in $\largestKnorm{\sumOfSquaredDifferences{\tu}}{K}[\nu]$, the vector~$\lambda$ associated with $\tu$ is given as $\lambda = 2 \, \rho \, (\tM - \tM_I) \, \tu$.
\end{lemma}
\begin{proof}
	We apply \cref{definition:critical-point} to problem \eqref{eq:general-problem:discrete:penalized:restated-fully} with
	\begin{align*}
		g(\tu)
		&
		=
		f(\tu) + \rho \, \norm{\sumOfSquaredDifferences{\tu}}_{1, \nu}
		=
		f(\tu) + \rho \, \tu^\transp \tM \tu
		,
		\\
		h(\tu)
		&
		=
		\rho \, \largestKnorm{\sumOfSquaredDifferences{\tu}}{K}[\nu]
		.
	\end{align*}
	Due to the differentiability of~$f$, and therefore~$g$, we have that $\partial g(\tu) = \set{\nabla f(\tu) + 2 \, \rho \, \tM \tu}$, so the claim for criticality follows.

	To show the equivalence of \cref{item:general-problem:discrete:penalized:multiplier:1,item:general-problem:discrete:penalized:multiplier:2}, let $\tu$ be a critical point.
	Strong criticality of~$\tu$ implies that $\partial h(\tu) = \rho \, \partial \largestKnorm{\sumOfSquaredDifferences{\cdot}}{K}[\nu](\tu)$ is a singleton, as $\partial g(\tu)$ is a singleton.
	On the other hand, if the subdifferential $\partial h(\tu)$ is a singleton and has nonempty intersection with the singleton $\partial g(\tu)$, then they coincide, which shows strong criticality of~$\tu$.

	To verify the form of the associated vector~$\lambda$, note that, using \cref{theorem:subdifferential-of-the-largest-K-norm-of-the-sum-of-squared-differences}, we can conclude that
	\begin{multline*}
		\set{\nabla f(\tu) + 2 \, \rho \, \tM \tu}
		=
		\partial g(\tu)
		=
		\partial h(\tu)
		=
		\rho \, \partial \largestKnorm{\sumOfSquaredDifferences{\cdot}}{K}[\nu](\tu)
		\\
		=
		\rho \,
		\conv
		\setDef[Big]
		{2 \, \tM_I \tu}
		{%
			I \subseteq \set{1, \ldots, \nsimplices}
			,
			\;
			\smash{\sum_{i \in I}} \nu_i \le K
			,
			\;
			\smash{\sum_{i \in I}} \nu_i \, \sumOfSquaredDifferences{\tu}_i = \largestKnorm{\sumOfSquaredDifferences{\tu}}{K}[\nu]
		}
		.
	\end{multline*}
	This means that $2 \, \tM_I \tu$ is the same element for all relevant index sets~$I$, because otherwise the set would not be a singleton, which yields the claimed expression.
\end{proof}

The general convergence results for the DC~algorithm, see \cite[Theorem~3]{PhamLeThi:1997:1}, imply the following convergence behavior of \cref{algorithm:DC-algorithm} for problem \eqref{eq:general-problem:discrete:penalized:restated-fully}.
\begin{theorem}[Convergence of the \namedref{algorithm:DC-algorithm}{DC Algorithm}] \skipline
	\label{theorem:DC-algorithm:convergence}
	Let $\sequence[big](){\tu}{k}$ and $\sequence[big](){\ts}{k}$ be the sequences generated by \cref{algorithm:DC-algorithm} for the penalized problem \eqref{eq:general-problem:discrete:penalized:restated-fully}.
	Then:
	\begin{enumerate}
		\item \label[statement]{item:DC-algorithm:convergence:1}
			The sequence $(f_\rho \sequence[big](){\tu}{k})$ of function values is monotonically decreasing.

		\item \label[statement]{item:DC-algorithm:convergence:2}
			If $\sequence{\tu}{k+1} = \sequence{\tu}{k}$, the current iterate $\sequence{\tu}{k}$ is a critical point for \eqref{eq:general-problem:discrete:penalized:restated-fully}.

		\item \label[statement]{item:DC-algorithm:convergence:3}
			If $f$ is strongly convex and $f_\rho(\sequence{\tu}{k+1}) = f_\rho \sequence[big](){\tu}{k}$, it holds $\sequence{\tu}{k+1} = \sequence{\tu}{k}$ and therefore $\sequence{\tu}{k}$ is a critical point for \eqref{eq:general-problem:discrete:penalized:restated-fully}.

		\item \label[statement]{item:DC-algorithm:convergence:4}
			Every accumulation point $\tu$ of $\sequence[big](){\tu}{k}$ is a critical point for \eqref{eq:general-problem:discrete:penalized:restated-fully}.
	\end{enumerate}
\end{theorem}
\begin{proof}
	\Cref{item:DC-algorithm:convergence:1} follows from \cite[Theorem~3~(i)]{PhamLeThi:1997:1} and \cref{item:DC-algorithm:convergence:2,item:DC-algorithm:convergence:3} are a consequence of subsequent statements in the same source.
	This leaves \cref{item:DC-algorithm:convergence:4} to be discussed.
	By \cite[Theorem~3~(iv)]{PhamLeThi:1997:1} it suffices to show boundedness of the sequences $\sequence[big](){\tu}{k}$ and $\sequence[big](){\ts}{k}$.
	Note that the properties of~$f$, see \cref{assumption:standing-assumptions:finite-dimensions:differentiability-and-convexity}, together with $\phi(\tu) \ge 0$ for all $\tu \in \R^\nvertices$, imply that the objective function $f_\rho$ is bounded from below and radially unbounded.
	Therefore, the sequence $\sequence[big](){\tu}{k}$ is bounded due to \cref{item:DC-algorithm:convergence:1}.
	Boundedness of $\sequence[big](){\tu}{k}$ and our choice of the subgradient $\sequence{\ts}{k}$, see \eqref{eq:subdifferential-of-the-largest-K-norm-of-the-sum-of-squared-differences:choice} and \eqref{eq:choice_of_subgradient}, then imply that $\sequence[big](){\ts}{k}$ is bounded as well since $\norm{\sequence{\tr}{k}}_\infty \le K$ for all $k$.
\end{proof}
The notion of (strong) criticality relates to first order necessary conditions for local optimality, therefore critical points can only be expected to be local minimizers in the absence of additional assumptions that generalize convexity, and in practice, convergence to global optimizers is not typically observed, \cf \cref{subsubsection:K-scheduling}.

Interestingly, the criticality conditions can be stated in a nodal manner, similarly to the nodal characterization of B-stationarity in \cref{theorem:B-stationarity:explicit-description} and the following results in \cref{subsection:optimality-conditions}, but is typically weaker in the sense that strong criticality is not equivalent to the primal notion of B-stationarity.

In contrast to the previous results of this subsection, which only rely on \cref{assumption:w-matrix-structure}, such a nodal description requires the specific structure of $\tD_1$ and $\tD_2$ arising from the finite element discretization.
For all finite element based settings, we therefore set the matrix $\tD_1 \in \R^{\nedges \times \nvertices}$ by
\begin{equation}
	\label{eq:D1:finite-elements}
	(\tD_1)_{j,\ell}
	\coloneqq
	\begin{cases}
		\mrep[r]{1}{-1}
		,
		&
		\text{if vertex $\ell$ is the first vertex of edge $j$}
		,
		\\
		-1
		,
		&
		\text{if vertex $\ell$ is the second vertex of edge $j$}
		,
		\\
		\mrep[r]{0}{-1}
		,
		&
		\text{else}
		.
	\end{cases}
\end{equation}
Consequently, given $\tu \in \R^\nvertices$, the quantity $(\tD_1 \tu)_j$, $j = 1, \ldots, \nedges$, represents the difference of vertex values incident on the $j$-th edge.
The sign, \ie, whether a vertex is considered first or second, is not essential here, since we take the square of this quantity in \eqref{eq:sum-of-squared-differences}.
Furthermore, a matrix $\tD_2 \in \R^{\nsimplices \times \nedges}$ that encodes whether an edge is part of a cell is defined by
\begin{equation}
	\label{eq:D2:finite-elements}
	(\tD_2)_{i,j}
	\coloneqq
	\begin{cases}
		1
		,
		&
		\text{if $j \in \simplextoedges{i}$, \ie, edge~$j$ is an edge of the cell $S_i$}
		\\
		0
		,
		&
		\text{else}
		.
	\end{cases}
\end{equation}
With these matrices, $\tw$ is given by \eqref{eq:sum-of-squared-differences}, \ie, $\sumOfSquaredDifferences{\tu}_i = \sum_{j \in \simplextoedges{i}} (\tD_1 \tu)_j^2$ is the sum of the squared differences of the vertex values on the cell~$S_i$.
Since $\tD_2$ has nonnegative entries, \cref{assumption:w-matrix-structure} is satisfied.
Moreover, $\tw$ is continuous, and $\tu \mapsto \largestKnorm{\sumOfSquaredDifferences{\tu}}{K}[\nu]$ is convex as the maximum of the convex quadratic functions $\tu \mapsto \tu^\transp \tM_I \tu$; see \cref{lemma:norm-of-sum-of-squared-differences-via-matrices} and the proof of \cref{theorem:subdifferential-of-the-largest-K-norm-of-the-sum-of-squared-differences}.
The remaining property in \cref{assumption:abstract-w-properties} is verified in the following lemma.

\begin{lemma}
	\label{lemma:constant-functions-yield-zero-sum-of-squared-differences}
	Let $\tD_1$ and $\tD_2$ be given by \eqref{eq:D1:finite-elements} and \eqref{eq:D2:finite-elements}, respectively.
	Then a function $u_h \in U_h$ is constant on~$S_i$ if and only if $\sumOfSquaredDifferences{\tu}_i = 0$ holds for its coefficient vector~$\tu$.
\end{lemma}
\begin{proof}
	Let $u_h \in U_h$ be constant on~$S_i$.
	Then by definition of $\tD_1$ it holds $(\tD_1 \tu)_j = 0$ for all $j \in \simplextoedges{i}$.
	By definition of $\tD_2$ and \eqref{eq:sum-of-squared-differences} this implies $\sumOfSquaredDifferences{\tu}_i = 0$.
	Conversely, let $\sumOfSquaredDifferences{\tu}_i = 0$ hold, \ie, by \eqref{eq:sum-of-squared-differences},
	$
	0
	=
	\paren[Big][]{\tD_2 \paren[big][]{(\tD_1 \tu) \odot (\tD_1 \tu)}}_i
	=
	\sum_{{j \in \simplextoedges{i}}{}} \paren[big][]{\tD_1 \tu}_j^2
	.
	$
	Clearly, this yields that the difference of vertex values incident on the $j$-th edge, $(\tD_1 \tu)_j$, has to vanish for all $j \in \simplextoedges{i}$.
	Since the edges of $S_i$ connect all of its vertices, this can only be true if all vertices of simplex $i$ share the same value, \ie, $u_h \in U_h$ is constant on~$S_i$.
\end{proof}

\begin{theorem}[Necessary conditions for strong criticality] \skipline
	\label{theorem:discrete-optimality-conditions:penalized}
	Let $\tD_1$ and $\tD_2$ be given by \eqref{eq:D1:finite-elements} and \eqref{eq:D2:finite-elements}, respectively.
	Let $\tu$ be a strongly critical point of $f_\rho$ \wrt the natural DC decomposition in \eqref{eq:general-problem:discrete:penalized:restated-fully} and let $I \subseteq \set{1, \ldots, \nsimplices}$ be an index set that realizes the maximum in $\largestKnorm{\sumOfSquaredDifferences{\tu}}{K}[\nu]$ that we assume \wolog to satisfy $I \subseteq \nonconstantsimplices{\tu}$.
	Then the following are true:
	\begin{enumerate}
		\item \label[statement]{item:discrete-optimality-conditions:penalized:6}
			$\sum_{\ell \in V} \nabla f(\tu)_\ell = 0$ for all $V \subseteq \set{1, \ldots, \nvertices}$ such that
			\begin{math}
				\vertextosimplices{V}\setminus\coveredsimplices{V}
				\subseteq
				I
				.
			\end{math}

		\item \label[statement]{item:discrete-optimality-conditions:penalized:4}
			$\nabla f(\tu)_\ell = 0$ for all $\ell\in \set{1, \ldots, \nvertices}$ such that $\vertextosimplices{\ell} \subseteq I$.

		\item \label[statement]{item:discrete-optimality-conditions:penalized:5}
			$\nabla f(\tu)_\ell = 0$ for all $\ell\in \set{1, \ldots, \nvertices}$ such that $\vertextosimplices{\ell} \subseteq \constantsimplices{\tu}$.
	\end{enumerate}
\end{theorem}
\begin{proof}
	Note that we can assume $I \subseteq \nonconstantsimplices{\tu}$ because $\sumOfSquaredDifferences{\tu}_i = 0$ holds for all $i \in \constantsimplices{\tu}$ by \cref{assumption:abstract-w-properties}.
	Since $\tu$ is a strongly critical point, \cref{lemma:general-problem:discrete:penalized:multiplier} yields
	\begin{equation}
		\label{eq:vector-gradient-form}
		\begin{aligned}
			- \nabla f(\tu)
			=
			\lambda
			=
			2 \, \rho \, (\tM - \tM_I) \, \tu
			&
			=
			2 \, \rho \, \tD_1^\transp \diag \paren[big](){\tD_2^\transp \nu - \tD_2^\transp \indicatorFunction{I} \nu} \, \tD_1 \tu
			\\
			&
			=
			2 \, \rho \, \tD_1^\transp \diag \paren[big](){\tD_2^\transp \chi_{\complement{I}} \nu} \, \tD_1 \tu
			.
		\end{aligned}
	\end{equation}
	Written in a nodal (\ie, componentwise) sense, this is exactly the statement that
	\begin{equation}
		\label{eq:nodal-gradient-form}
		- \nabla f(\tu)_\ell
		=
		2 \, \rho \sum_{j=1}^\nedges (\tD_1)_{j,\ell} \paren[auto](){\tD_2^\transp \chi_{\complement{I}} \nu}_j (\tD_1 \tu)_j
		\quad
		\text{for all }
		\ell = 1, \ldots, \nvertices
		.
	\end{equation}
	\Cref{item:discrete-optimality-conditions:penalized:6}:
	Let $V \subseteq \set{1, \ldots, \nvertices}$ with $\vertextosimplices{V}\setminus\coveredsimplices{V} \subseteq I$ be given.
	We define
	\begin{equation*}
		J
		\coloneqq
		\setDef[auto]{j \in \set{1, \ldots, \nedges}}{\cardinality{V \cap \edgetovertices{j}} = 1}
	\end{equation*}
	and note that
	\begin{math}
		\edgetosimplices{J}
		=
		\vertextosimplices{V} \setminus \coveredsimplices{V}
		,
	\end{math}
	because these are exactly the simplices that have at least one but not all vertices in~$V$.

	Using the definitions of $\tD_1$ and $\tD_2$, we obtain that
	\begin{equation*}
		\sum_{\ell \in V} (\tD_1)_{j,\ell}
		=
		0
		\text{ for all }
		j \in \set{1, \ldots, \nedges} \setminus J
		\quad
		\text{and}
		\quad
		\sum_{i \in \complement{I}} (\tD_2)_{i,j} \nu_i
		=
		0
		\text{ for all }
		j \in J
		,
	\end{equation*}
	where the left-hand sum is either a sum of zeros or a sum of zeros and $\pm 1$, and where in the right-hand sum, the entries $(\tD_2)_{i,j}$ are all zero because the edges $j \in J$ are exactly the ones where the simplices these edges are incident to are in $\vertextosimplices{V}\setminus\coveredsimplices{V}$ and therefore in $I$.

	Using \eqref{eq:nodal-gradient-form}, we therefore have that
	\begin{align*}
		- \frac{1}{2 \, \rho}\sum_{l \in V} \nabla f(\tu)_\ell
		&
		=
		\sum_{l \in V} \sum_{j=1}^\nedges
		(\tD_1)_{j,\ell}
		\paren[big](){\tD_2^\transp \chi_{\complement{I}} \nu}_j
		(\tD_1 \tu)_j
		\\
		&
		=
		\sum_{j=1}^\nedges
		\paren[Big](){\sum_{l \in V} (\tD_1)_{j,\ell}}
		\paren[Big](){\sum_{i \in \complement{I}} (\tD_2)_{i,j} \nu_i}
		(\tD_1 \tu)_j
		=
		0
		.
	\end{align*}

	\Cref{item:discrete-optimality-conditions:penalized:4} is an immediate consequence of \cref{item:discrete-optimality-conditions:penalized:6}.

	\Cref{item:discrete-optimality-conditions:penalized:5}:
	Let $\ell \in \set{1, \ldots, \nvertices}$ with $\vertextosimplices{\ell} \subseteq \constantsimplices{\tu}$.
	Then, by definition, $(\tD_1 \tu)_j = 0$ for all $j \in \vertextoedges{\ell}$.
	Because $(\tD_1)_{j,\ell} = 0$ for every $j \in \set{1, \ldots, \nedges} \setminus \vertextoedges{\ell}$, this implies $(\tD_1)_{j,\ell} (\tD_1 \tu)_j  = 0$ for every $j \in \set{1, \ldots, \nedges}$.
	Again, \eqref{eq:nodal-gradient-form} implies $\nabla f(\tu)_\ell = 0$.
\end{proof}

Note that since the above result states optimality conditions that were generated using penalization techniques, there is no direct connection to feasibility of the points in question coming into play.
This becomes especially apparent when comparing \cref{theorem:discrete-optimality-conditions:penalized}~\ref{item:discrete-optimality-conditions:penalized:6} to the nodal characterization of B-stationarity for the discretized non-penalized problem, see \cref{theorem:B-stationarity:explicit-description}, which has similar vanishing average gradient structure and involves the same set of non-covered cells but where the information is given for differently structured subsets of vertices.
Additionally, we have the following result.

\begin{corollary}
	\label{corollary:discrete-optimality-conditions:penalized:exact-penalization}
	Let $\tu$ be a strongly critical point of $f_\rho$ \wrt the natural DC decomposition in \eqref{eq:general-problem:discrete:penalized:restated-fully} such that $\norm{\sumOfSquaredDifferences{\tu}}_{0,\nu} \le K$.
	Then $\nabla f(\tu) = 0$.
\end{corollary}
\begin{proof}
	If $\norm{\sumOfSquaredDifferences{\tu}}_{0,\nu} \le K$, then $\tu^\transp (\tM - \tM_I) \, \tu = \norm{\sumOfSquaredDifferences{\tu}}_{1, \nu} - \largestKnorm{\sumOfSquaredDifferences{\tu}}{K}[\nu] = 0$ and analogously to \eqref{eq:vector-gradient-form} and \eqref{eq:nodal-gradient-form} we therefore have that
	\begin{equation}
		0
		=
		\tu^\transp (\tM - \tM_I) \, \tu
		=
		\sum_{j=1}^\nedges \paren[auto](){\tD_2^\transp \chi_{\complement{I}} \nu}_j (\tD_1 \tu)_j^2
		,
	\end{equation}
	\ie, $\paren[auto](){\tD_2^\transp \chi_{\complement{I}} \nu}_j (\tD_1 \tu)_j = 0$ for all $j = 1, \ldots, \nedges$.
	Then $(\tM - \tM_I) \, \tu = 0$ and \eqref{eq:vector-gradient-form} yields the claim.
\end{proof}

\Cref{corollary:discrete-optimality-conditions:penalized:exact-penalization} implies that exact penalization is generally impossible.
That is, even for large values of the penalty parameter $\rho > 0$, a solution of the penalized problem \eqref{eq:general-problem:discrete:penalized:restated-fully} cannot be a solution of the original problem \eqref{eq:general-problem:discrete}.
The only exception occurs in case the cardinality constraint is inactive at the solution and therefore, penalization is unnecessary.
Likewise, strongly critical points of \eqref{eq:general-problem:discrete:penalized:restated-fully} are, in general, not $B$-stationary points of \eqref{eq:general-problem:discrete}.

On the other hand, we remind the reader of \cref{theorem:general-problem:discrete:penalized:convergence} which states that for large values of $\rho > 0$, the minimizer of \eqref{eq:general-problem:discrete:penalized:restated-fully} is close to the minimizer of the original problem \eqref{eq:general-problem:discrete}.

\section{Numerical Results}
\label{section:numerical-results}

This section presents the numerical results of applying the outlined DC-reformulation approach to three types of problems -- a classical energy minimization problem, an optimal control problem, and an image reconstruction problem.

In both of the following infinite dimensional examples we will need at least $(1+\varepsilon)$-regularity in order to get existence of solutions, see \cref{assumption:standing-assumptions} and \cref{example:standing-assumptions}.
However, our discrete finite-element-space cannot provide this regularity for $\varepsilon$ being an integer, which is why we use a fractional Sobolev space.

For $\varepsilon \in (0,1)$, $k \in \N$ and open $\Omega \subseteq \R^d$ the fractional Sobolev spaces $H^\varepsilon(\Omega)$ and $H^{k+\varepsilon}(\Omega)$ are defined by
\begin{align*}
	H^\varepsilon(\Omega)
	&
	\coloneqq
	\setDef{v \in L^2(\Omega)}{\abs{v}_{H^\varepsilon(\Omega)} < \infty}
	,
	\\
	H^{k+\varepsilon}(\Omega)
	&
	\coloneqq
	\setDef{v \in H^k(\Omega)}{\abs{D^\alpha v} \in H^\varepsilon(\Omega) \text{ for all } \alpha \text{ with } \abs{\alpha} = k}
\end{align*}
with the Aronszajn-Slobodeckij seminorm
\begin{equation*}
	\abs{v}_{H^\varepsilon(\Omega)}^2
	\coloneqq
	\int_{\Omega^2} \frac{\abs{v(x)-v(y)}^2}{\abs{x-y}^{d+2\varepsilon}} \d(x,y)
	.
\end{equation*}

For a piecewise affine and continuous function $u_h \in U_h \subseteq H^1(\Omega)$ the weak derivatives $\frac{\partial u_h}{\partial x_i}$ are piecewise constant functions and therefore $\frac{\partial u_h}{\partial x_i} \in H^\varepsilon(\Omega)$ for $\varepsilon \in (0,\frac12)$, see \eg, \cite[p.578]{Fusco:2015:1} or \cite[Corollary~2.2~(i)]{Lombardini:2018:1}.
Therefore, $U_h \subseteq H^{k+\varepsilon}(\Omega)$ for $\varepsilon \in (0,\frac12)$ and hence $U_h$ provides a conforming discretization.

Moreover, denoting the value of $\frac{\partial u_h}{\partial x_i}$ on $S_j$ by $(\partial_i u_h)_j$ for $u_h \in U_h$ and defining
\begin{equation*}
	\tW_{j,k} \coloneqq \int_{S_j \times S_k} \frac{1}{\abs{x-y}^{d+2\varepsilon}} \d(x,y) = \tW_{k,j} \quad \text{for } j \neq k \quad \text{and } \quad \tW_{j,j} \coloneqq 0
\end{equation*}
results in
\begin{align*}
	\abs{\frac{\partial u_h}{\partial x_i}}_{H^\varepsilon(\Omega)}^2
	&
	=
	\sum_{j,k=1}^\nsimplices \int_{S_j \times S_k} \frac{\abs{\frac{\partial u_h}{\partial x_i}(x)-\frac{\partial u_h}{\partial x_i}(y)}^2}{\abs{x-y}^{d+2\varepsilon}} \d(x,y)
	=
	\sum_{\substack{j,k=1 \\ j \neq k}}^\nsimplices \left((\partial_i u_h)_j-(\partial_i u_h)_k\right)^2 \tW_{j,k}
	\\
	&
	=
	\sum_{j,k=1}^\nsimplices (\partial_i u_h)_j^2 \tW_{j,k} + \sum_{j,k=1}^\nsimplices (\partial_i u_h)_k^2 \tW_{j,k} - 2\sum_{j,k=1}^\nsimplices (\partial_i u_h)_j (\partial_i u_h)_k \tW_{j,k}
	\\
	&
	=
	2 (\partial_i u_h)^\transp \left(\diag\left((\sum_{k=1}^\nsimplices \tW_{j,k})_j\right) - \tW \right) (\partial_i u_h)
	=
	(\partial_i u_h)^\transp \tilde{\tW} (\partial_i u_h)
\end{align*}
with the matrix $\tilde{\tW} \coloneqq 2\left(\diag\left((\sum_{k=1}^\nsimplices \tW_{j,k})_j\right) - \tW \right)$, which has the form of a graph laplacian.

The computation of the integrals $\tW_{j,k}$ for the assembly of the dense matrix $\tW$ is done analogously to the computation of the integrals $I_{\ell,m}^{i,j}$ in \cite[Equation~(3.3)]{AcostaBersetcheBorthagaray:2017:1}, where the difficult cases are treated in \cite[Appendices~A.2, A.3]{AcostaBersetcheBorthagaray:2017:1}.
$\partial_i u_h$ can be calculated explicitly via $\partial_i u_h = \tG_i \tu$ with a discrete linear differential operator $\tG_i$.

Additionally, we will be needing to specify the matrices making up $\tw$ for each application.
For all finite element based applications, we use the matrices $\tD_1$ and $\tD_2$ defined in \eqref{eq:D1:finite-elements} and \eqref{eq:D2:finite-elements}, respectively.

\subsection{Energy Minimization}
\label{subsection:numerical-results:energy-optimization}

The problem analyzed in this section is given as
\begin{equation}
	\label{example:general-problem}
	\begin{aligned}
		\text{Minimize}
		\quad
		&
		\frac{1}{2} \norm{\nabla u}_{L^2(\Omega)}^2
		-
		\int_\Omega g \, u \d \mu
		+
		\frac{\alpha}{2}\sum_{i=1}^2 \abs{\frac{\partial u}{\partial x_i}}_{H^\varepsilon(\Omega)}^2
		\\
		\text{\st}
		\quad
		&
		\norm{\nabla u}_0 \le K
		\quad
		\text{where }
		u \in U \coloneqq H_0^1(\Omega) \cap H^{1+\varepsilon}(\Omega)
		.
	\end{aligned}
\end{equation}
We choose $\varepsilon \coloneqq \frac14$ and set $\Omega \coloneqq \interval(){0}{1}^2$ and
\begin{math}
	g(x,y)
	\coloneqq
	10 \, x \, \sin(5 \, x) \, \sin(7 \, y)
	.
\end{math}
Note that \eqref{example:general-problem} satisfies \cref{assumption:standing-assumptions} and is therefore globally solvable by \cref{theorem:general-problem:penalized:existence}.
Numerically, we incorporate the boundary conditions directly into the discretization and let $\ninnervertices$ and $\ninnersimplices$ denote the number of interior vertices and the number of simplices with at least one interior vertex, respectively.
Rows and columns associated with quantities on the boundary are removed from $\tD_1$ and $\tD_2$.
The associated discrete and penalized problem is then of the general form \eqref{eq:general-problem:discrete:penalized:restated-fully} and reads
\begin{equation*}
	\text{Minimize}
	\quad
	\frac{1}{2} \tu^\transp \tA \tu - \tb^\transp \tu
	+
	\frac{\alpha}{2} \tu^\transp \tL \tu
	+
	\rho \, \paren[big](){\norm{\sumOfSquaredDifferences{\tu}}_{1,\nu} - \largestKnorm{\sumOfSquaredDifferences{\tu}}{K}[\nu]}
	\quad
	\text{where }
	\tu \in \R^\ninnervertices
	.
\end{equation*}
Here $\tA$ and $\tb$ are the stiffness matrix and load vector associated with the finite element basis $\family{\varphi_i}_{i=1}^\ninnervertices$, \ie,
\begin{alignat*}{2}
	\tA
	&
	\in \R^{\ninnervertices \times \ninnervertices}
	&
	&
	\quad
	\text{with}
	\quad
	\tA_{i,j}
	\coloneqq
	\int_\Omega \nabla \varphi_i \cdot \nabla \varphi_j \d \mu
	\\
	\text{and}
	\quad
	\tb
	&
	\in \R^\ninnervertices
	&
	&
	\quad
	\text{with}
	\quad
	\tb_\ell
	\coloneqq
	\int_\Omega g \, \varphi_\ell \d \mu
	,
\end{alignat*}
and $\tL \coloneqq \tG_1^\transp \tilde{\tW} \tG_1 + \tG_2^\transp \tilde{\tW} \tG_2$.

We have implemented \cref{algorithm:DC-algorithm} in \python.
For mesh generation with a target mesh size $h$, we used \gmsh \cite{GeuzaineRemacle:2009:1}.
Finite element assembly was achieved using the \scikitfem library \cite{GustafssonMcBain:2020:1}.
The mesh sizes, \ie, maximal edge lengths $\max \setDef[big]{\diam(S_i)}{i \in \set{1, \ldots, \ninnersimplices}}$, as well as the number of unknowns are reported in \cref{table:problem-dimensions}.

\begin{table}[htp]
	\centering
	\caption{Reciprocal of target mesh size, actual mesh size (maximal edge length) and number of unknowns.}
	\label{table:problem-dimensions}
	\begin{tabular}{rrr@{\qquad\qquad}rrr}
		\toprule
		$1/h$
		&
		mesh size
		&
		unknowns
		&
		$1/h$
		&
		mesh size
		&
		unknowns
		\\
		\midrule
		\num{32}
		&
		$4.0 \cdot 10^{-2}$
		&
		\num{1137}
		&
		\num{256}
		&
		$5.1 \cdot 10^{-3}$
		&
		\num{75355}
		\\
		\num{64}
		&
		$1.9 \cdot 10^{-2}$
		&
		\num{4631}
		&
		\num{512}
		&
		$2.6 \cdot 10^{-3}$
		&
		\num{302227}
		\\
		\num{128}
		&
		$1.0 \cdot 10^{-2}$
		&
		\num{18730}
		&
		\num{1024}
		&
		$1.3 \cdot 10^{-3}$
		&
		\num{1209308}
		\\
		\bottomrule
	\end{tabular}
\end{table}

Unless mentioned otherwise, we initialize the \namedref{algorithm:DC-algorithm}{DC Algorithm} with $\sequence{\tu}{0} = \tA^{-1} \tb$, \ie, the solution of the unconstrained problem.
We stop the algorithm as soon as the decrease in function value and the feasibility penalty term reach the desired tolerance, \ie,
\begin{equation}
	\label{eq:stopping-criterion}
	\begin{aligned}
		\abs[big]{f(\sequence{\tu}{k}) - f(\sequence{\tu}{k+1})}
		&
		<
		10^{-15}
		\\
		\text{and}
		\quad
		\psi(\sequence{\tu}{k+1})
		\coloneqq
		\norm[big]{\tD_2 \abs{\tD_1 \sequence{\tu}{k+1}}}_{1, \nu}
		-
		\largestKnorm[big]{\tD_2 \abs{\tD_1 \sequence{\tu}{k+1}}}{K}[\nu]
		&
		<
		10^{-15}
		.
	\end{aligned}
\end{equation}

\begin{remark}
	Note that the criterion based on $\psi$ is less prone to numerical cancellation than a criterion based on the penalty term
	\begin{equation*}
		\phi(\tu)
		=
		\norm{\sumOfSquaredDifferences{\tu}}_{1,\nu}
		-
		\largestKnorm{\sumOfSquaredDifferences{\tu}}{K}[\nu]
	\end{equation*}
	because $\phi(\tu)$ consists of a sum of \emph{squared} differences tending to zero.
	The criterion based on $\psi$ equivalently determines feasibility since $\psi(\tu) = 0 \Leftrightarrow \phi(\tu) = 0$; see also \cref{remark:not-squaring-the-differences-of-vertex-values}.
\end{remark}

We employed \scipy's \texttt{spsolve} function to solve linear systems of equations, primarily those in \cref{algorithm:DC-algorithm:solve-subproblem} of \cref{algorithm:DC-algorithm}.
Moreover, if not mentioned otherwise, we used
\begin{equation}
	\label{eq:standard-parameters}
	K = 0.1
	,
	\quad
	\rho = 10^{18}
	,
	\quad
	h = \frac{1}{256}
	,
\end{equation}
where the magnitude of $K$ was chosen such that localizing sparsity effects could be observed visually, \cf
\cref{subsubsection:numerical-results:energy-optimization:level-of-K} and the magnitude of the penalty parameter $\rho$ was chosen such that feasibility of the final iterate was reached up to machine precision, \cf \cref{subsubsection:numerical-results:energy-optimization:penalty-parameter}.

To determine $\largestKnorm{\sumOfSquaredDifferences{\sequence{\tu}{k}}}{K}[\nu]$ as well as a subgradient in \cref{algorithm:DC-algorithm:determine-subgradient} of \cref{algorithm:DC-algorithm}, a suitable index set for the weighted largest-$K$-norm is determined by a greedy algorithm for knapsack problems \cite[Section~2.4]{MartelloToth:1990:1}.

Because any numerical realization of the DC~algorithm will likely not produce results where the entries of $\sumOfSquaredDifferences{\tu}$ are exactly $0$ outside of what one would consider the support of the gradient of the solution, interpreting the support of the final iterates requires thresholding.
Specifically, we interpret every entry of $\sumOfSquaredDifferences{\tu}$ that is smaller than the threshold of $10^{-10}$ as~$0$, thus not contributing to the $L^0$-pseudo-norm of a computed solution's gradient.

In the following, we report on the influence of different parameters on the behavior of \cref{algorithm:DC-algorithm} with the subgradient choice \eqref{eq:choice_of_subgradient}.
In particular, we compare the computed results with respect to the functional value $f(\tu^*)$ at the final iterate $\tu^*$, the $\ell_0$-pseudo-norm
\begin{equation*}
	\norm{\nabla u^*_h}_0
	=
	\norm{\sumOfSquaredDifferences{\tu^*}}_{0,\nu}
	,
\end{equation*}
numerically realized as $	\sum_{i=1}^\ninnersimplices \nu_i \, \norm[big]{\max \set{\sumOfSquaredDifferences{\tu^*}_{i} - 10^{-10}, \; 0}}_0$,
the feasibility measured by $\psi(\tu^*)$ and $\phi(\tu^*)$
and the total number of DC iterations (DC).

Moreover, since the DC~algorithm converges to critical points, we plot $-\tB^{-1} \nabla f(\tu^*)$, \ie, the $L^2$-Riesz representative of the continuous quantity corresponding to the $\lambda$ associated with~$\tu^*$; \cf \eqref{eq:general-problem:discrete:penalized:multiplier}.
Here, $\tB$ is the mass matrix associated with the finite element basis $\family{\varphi_i}_{i=1}^\ninnervertices$.

\subsubsection{Dependence on the Fractional Term}
\label{subsubsection:Fractional}

Results for different values of the fractional regularization parameter $\alpha$ are presented in \cref{table:fractional}.
As the numerical results for $\alpha = 10^{-5}$ and $\alpha=0$ line up nearly perfectly, which we also observed qualitatively in the plots, we will neglect the fractional term in the following, \ie, set $\alpha = 0$.
This allows us to use finer discretizations as all involved matrices are then sparse in contrast to the dense matrix $\tL$.

\begin{table}[htp]
	\centering
	\caption{Results for different target mesh sizes~$h$ and regularization parameter~$\alpha$: functional value, functional value without fractional term, deviation of the cardinality constraint ($K = 0.1$), and number of DC iterations.}
	\label{table:fractional}
	\begin{filecontents*}{discretization_frac.csv}
1/h,alpha,function value,function value without Frac-term,L0-norm,K-L0,1-feas.,2-feas.,DC
32,$10^{-2}$,-0.001803666321472437,-0.008159481422634049,0.09981026669169729,0.00018973330830271,9.97465998686664e-16,3.3881317890172014e-21,313
,$10^{-3}$,-0.007713418020797663,-0.008824973064448821,0.0993701576971989,0.0006298423028011,9.783840404509192e-16,1.6940658945086007e-20,298
,$10^{-4}$,-0.008553592413842114,-0.008672795395297368,0.09948184558764298,0.00051815441235702,9.731798700229888e-16,-4.404571325722362e-20,317
,$10^{-5}$,-0.008837607292140814,-0.008849666110903258,0.09955109983272631,0.00044890016727369,9.675420187260642e-16,-2.371692252312041e-20,319
,0,-0.008849611916230234,-0.008849611916230234,0.0995510998327263,0.000448900167273711,9.71445146547012e-16,1.0164395367051604e-20,319
64,$10^{-2}$,-0.0023974537571899036,-0.00890303678098861,0.09965911780677933,0.00034088219322067,9.397864431104352e-16,-2.541098841762901e-21,241
,$10^{-3}$,-0.008620737453408434,-0.009728219587195419,0.0999648243355367,0.0000351756644633,9.204876444401933e-16,5.082197683525802e-21,244
,$10^{-4}$,-0.009605568569006165,-0.009722520068217616,0.09984892816411756,0.00015107183588244,9.339317513790135e-16,6.776263578034403e-21,254
,$10^{-5}$,-0.009664013094736315,-0.009675770149366537,0.09999926893078209,7.3106921791e-7,9.835882108788496e-16,-1.1011428314305904e-20,233
,0,-0.009685537346756961,-0.009685537346756961,0.099980801838803,1.91981611969982E-05,9.02706728811431e-16,-4.235164736271502e-21,233
128,$10^{-2}$,-0.0025849538274902864,-0.009337819696735129,0.09969454741020792,0.00030545258979208,9.38539610612077e-16,-4.235164736271502e-22,221
,$10^{-3}$,-0.008898036731208373,-0.01001841758986101,0.09975442131896824,0.00024557868103176,8.905636644795933e-16,1.4823076576950256e-21,238
,$10^{-4}$,-0.009939197481739414,-0.010058185688820562,0.09991250603151013,0.00008749396848987,9.653736143810931e-16,-4.446922973085077e-21,240
,$10^{-5}$,-0.010035567683344083,-0.010047496028592647,0.09988607703359897,0.00011392296640103,9.862987163100634e-16,6.352747104407253e-22,224
,0,-0.010065664343498638,-0.010065664343498638,0.0997274333819108,0.000272566618089193,9.621210078636366e-16,-4.235164736271502e-22,222
\end{filecontents*}
\pgfplotstabletypeset[
	col sep = comma,
	header = false,
	skip rows between index = {0}{1},
	sci,
	sci zerofill,
	column type = l,
	columns = {0,1,2,3,5,8},
	every head row/.style = {
		before row = \toprule,
		after row = \midrule,
	},
	every row no 5/.style = {
		before row = \midrule,
	},
	every row no 10/.style = {
		before row = \midrule,
	},
	every last row/.style = {
		after row = \bottomrule,
	},
	columns/0/.style = {
		column name = {$1/h$},
		column type = r,
		string type,
	},
	columns/1/.style = {
		column name = {$\alpha$},
		column type = r,
		string type,
	},
	columns/2/.style = {
		column name = {$f(\tu^*)$},
		dec sep align,
	},
	columns/3/.style = {
		column name = {$f(\tu^*)-\frac{\alpha}{2}(\tu^*)^\transp \tL \tu^*$},
		dec sep align,
	},
	columns/4/.style = {
		column name = {$\norm{\nabla \tu^*}_0$},
		dec sep align,
	},
	columns/5/.style = {
		column name = {$K-\norm{\nabla u_h^*}_0$},
		dec sep align,
	},
	columns/6/.style = {
		column name = {$\psi(\tu^*)$},
		dec sep align,
	},
	columns/7/.style = {
		column name = {$\phi(\tu^*)$},
		dec sep align,
	},
	columns/8/.style = {
		column name = {DC},
		column type = r,
		string type,
	},
	]{discretization_frac.csv}
\end{table}

\subsubsection{Dependence on the Initial Guess}
\label{subsubsection:K-scheduling}

Though \cref{algorithm:DC-algorithm} converges to a critical point from any initial guess $\tu_0$, in practice we observe a variance in objective function values of the resulting critical points when the initial guess is varied.
Specifically, we can improve the objective functional value of the critical point by employing a scheduling for the upper bound~$K$ on the $L^0$-pseudo-norm of the gradient, in our case using the update rule
\begin{equation}
	\label{eq:K-schedule}
	\begin{aligned}
		\sequence{K}{0}
		&
		\coloneqq
		\max \set{\beta \, \mu(\Omega), K}
		=
		\max \set{\beta, K}
		,
		\\
		\sequence{K}{k}
		&
		\coloneqq
		\max \set{\beta \sequence{K}{k-1}, K}
		\quad
		\text{for all }
		k \ge 1
		,
	\end{aligned}
\end{equation}
where $\beta$ is the scheduling parameter, \cf~\cite{DittrichWachsmuth:2025:1}.
\Ie, instead of computing the data $\sequence{\tr}{k}$ in \cref{algorithm:DC-algorithm:determine-subgradient} according to the target level $K$ from the beginning, we compute the data for a decreasing sequence $\sequence{K}{k}$ matching the target level $K$ after a finite number of iterations.
Subsequently, $K$ remains fixed.
Hence, the convergence of \cref{algorithm:DC-algorithm} is not affected.
\Cref{figure:solution:512} illustrates the computed solution with and without scheduling for $1/h = 512$.
We clearly observe the desired sparsity of the gradient.

In the setting without scheduling, we obtain $f(\tu^*) = -4.58 \cdot 10^{-3}$, $K - \norm{\nabla u_h^*}_0 = 5.46 \cdot 10^{-4}$, $\psi(\tu^*) = 9.95 \cdot 10^{-16}$ and $\phi(\tu^*) = 1.19 \cdot 10^{-22}$ after $\num{2561}$ DC~iterations.
Compared to our results reported with a $K$-schedule, \cf \cref{figure:solution:512} and \cref{table:level-of-discretization}, the solution without scheduling is clearly worse and is reached after a much larger number of iterations.

\begin{figure}[htbp]
	\centering
	\includegraphics[width = 0.45\linewidth]{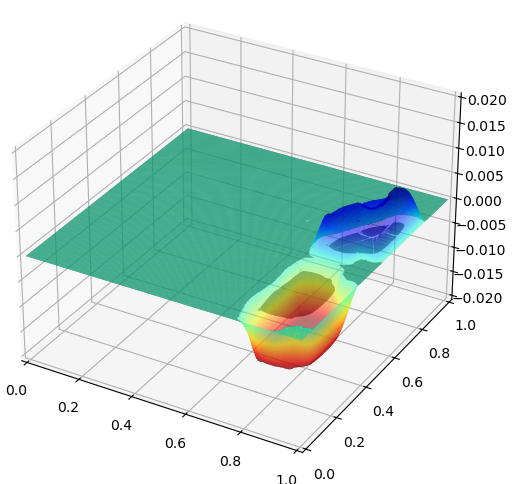}
	\hfill
	\includegraphics[width = 0.45\linewidth]{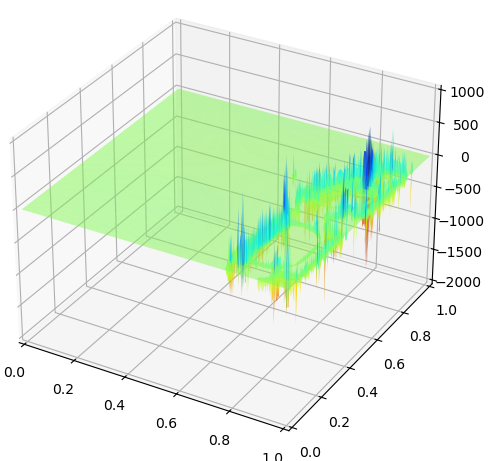}
	\\
	\includegraphics[width = 0.45\linewidth]{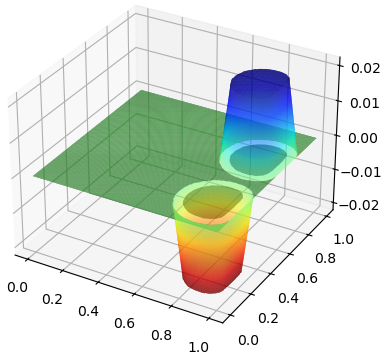}
	\hfill
	\includegraphics[width = 0.45\linewidth]{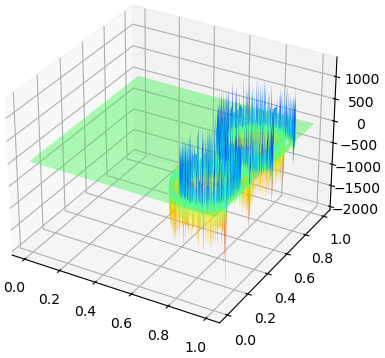}
	\caption{%
		Computed solution and $\lambda$, \cf \eqref{eq:general-problem:discrete:penalized:multiplier}, for $1/h = 512$, without using a $K$-schedule (top row), and with $K$-schedule (bottom row).
		Triangles where the solution is constant in the sense of our threshold are indicated by a darker shade.
	}
	\label{figure:solution:512}
\end{figure}

Therefore, the scheduling \eqref{eq:K-schedule} with $\beta = 0.95$ will be used in all upcoming results.

\subsubsection{Dependence on $K$}
\label{subsubsection:numerical-results:energy-optimization:level-of-K}

\Cref{figure:solution:256:sparsity-parameter} visualizes solutions for different target values of~$K$.
Setting $K = 1$ is equivalent to having no constraint, which is clearly observable in the solution with the solution's gradient being supported on the whole domain.
As $K$ decreases, we see the expected contraction of the support of the solution's gradient.

\begin{figure}[htbp]
	\centering
	\includegraphics[width = 0.45\linewidth]{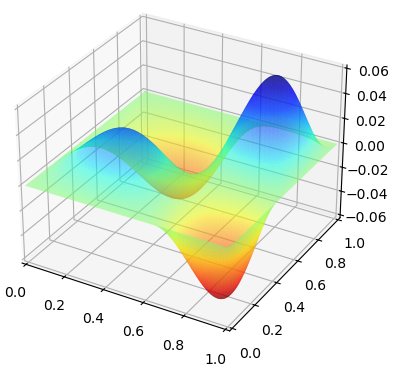}
	\hfill
	\includegraphics[width = 0.45\linewidth]{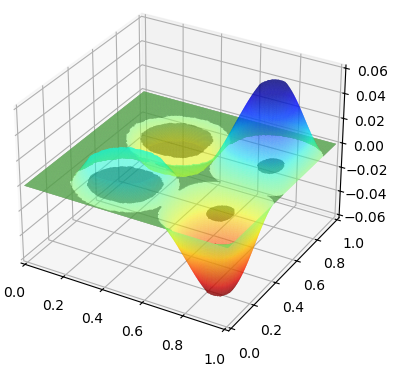}
	\\
	\includegraphics[width = 0.45\linewidth]{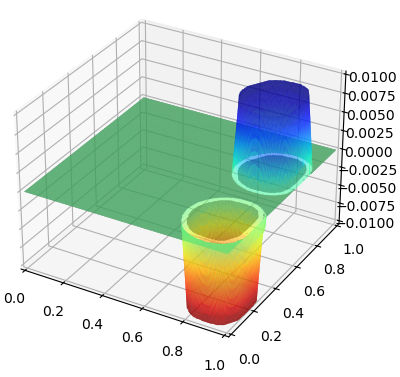}
	\hfill
	\includegraphics[width = 0.45\linewidth]{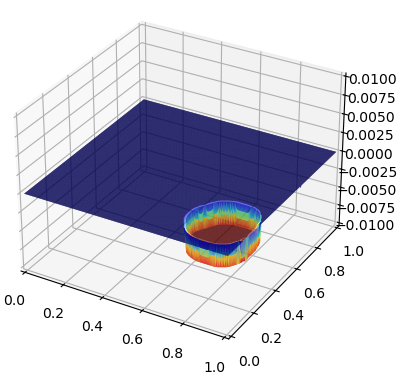}
	\caption{%
		Computed solutions for $K = 1, 0.5, 0.05, 0.01$ (top left to bottom right).
		Triangles where the solution is constant in the sense of our threshold are indicated by a darker shade.
		Note that the plot range in the top row is six times as large as in the bottom row.
	}
	\label{figure:solution:256:sparsity-parameter}
\end{figure}

\subsubsection{Dependence on the Discretization Level}
\label{subsubsection:numerical-results:energy-optimization:level-of-discretization}

In \cref{table:level-of-discretization}, we report results for various target mesh sizes, while the other parameters remain unchanged.
We tend to observe convergence in objective value for $h \searrow 0$, while there is no clear connection between the number of DC iterations and the level of discretization.

\begin{table}[htp]
	\centering
	\caption{Results for different target mesh sizes~$h$: functional value, deviation of the cardinality constraint ($K = 0.1$), feasibility and number of DC iterations.}
	\label{table:level-of-discretization}
	\begin{filecontents*}{discretization.csv}
1/h,function value,L0-norm,K-L0,1-feas.,2-feas.,DC
32,-0.008849611916230234,0.0995510998327263,0.000448900167273711,9.71445146547012e-16,1.0164395367051604e-20,319
64,-0.009685537346756961,0.099980801838803,1.91981611969982E-05,9.02706728811431e-16,-4.235164736271502e-21,233
128,-0.010065664343498638,0.0997274333819108,0.000272566618089193,9.621210078636366e-16,-4.235164736271502e-22,222
256,-0.010288800366395847,0.0999332479105502,6.67520894497842E-05,9.644520425344805e-16,-1.164670302474663e-21,232
512,-0.010399037401031395,0.0999896507194644,1.03492805356015E-05,9.721498779591276e-16,1.2176098616780567e-21,220
1024,-0.010381106293753948,0.0999797436861977,2.02563138023415E-05,5.704665255804042e-16,-5.624828165360588e-22,321
\end{filecontents*}
\pgfplotstabletypeset[
	col sep = comma,
	header = false,
	skip rows between index = {0}{1},
	sci,
	sci zerofill,
	column type = l,
	columns = {0,1,3,4,5,6},
	every head row/.style = {
		before row = \toprule,
		after row = \midrule,
	},
	every last row/.style = {
		after row = \bottomrule,
	},
	columns/0/.style = {
		column name = {$1/h$},
		column type = r,
		string type,
	},
	columns/1/.style = {
		column name = {$f(\tu^*)$},
		dec sep align,
	},
	columns/2/.style = {
		column name = {$\norm{\nabla \tu^*}_0$},
		dec sep align,
	},
	columns/3/.style = {
		column name = {$K-\norm{\nabla u_h^*}_0$},
		dec sep align,
	},
	columns/4/.style = {
		column name = {$\psi(\tu^*)$},
		dec sep align,
	},
	columns/5/.style = {
		column name = {$\phi(\tu^*)$},
		dec sep align,
	},
	columns/6/.style = {
		column name = {DC},
		column type = r,
		string type,
	},
	]{discretization.csv}
\end{table}

\subsubsection{Dependence on the Penalty Parameter}
\label{subsubsection:numerical-results:energy-optimization:penalty-parameter}

In this section, we elaborate on the choice of the rather large penalty parameter of $\rho = 10^{18}$ as a default parameter value.

\begin{table}[htp]
	\centering
	\caption{Results for different values of the penalty parameter $\rho$: functional value, deviation of the cardinality constraint ($K = 0.1$), feasibility and number of DC iterations.}
	\label{table:penalty-parameter}
	\begin{filecontents*}{penalty.csv}
rho,function value,L0-norm,K-L0,1-feas.,2-feas.,DC
10^3,-0.035160166292531464,0.999913173820185,-0.899913173820185,0.0012824295540279252,9.032967951111927e-07,44
10^4,-0.03322927953429836,0.999873929784494,-0.899873929784494,0.0009517307931545782,5.06245547181168e-07,46
10^5,-0.019331927761749447,0.998164734562386,-0.898164734562386,0.00026251783396667874,4.009141851650676e-08,49
10^6,-0.011516921698395535,0.895780673516351,-0.795780673516351,3.2178511221103484e-05,6.194389872737979e-10,60
10^7,-0.010416841132143516,0.100319391167169,-0.000319391167168975,3.2905459995991883e-06,6.484762210606965e-12,73
10^8,-0.01029536124686254,0.0999409501339778,5.90498660222316E-05,3.30039469378983e-07,6.503172193729912e-14,89
10^9,-0.01028805397510276,0.099936934130814,6.30658691860575E-05,3.295062136397754e-08,6.490530777563557e-16,98
10^{10},-0.010283716883875879,0.0998842130923492,0.00011578690765085,3.3015628632697067e-09,6.492613419822619e-18,114
10^{11},-0.010288814051033917,0.0999332479105502,6.67520894497842E-05,3.290179345335001e-10,6.336865236646234e-20,141
10^{12},-0.010288801768038904,0.0999332479105502,6.67520894497842E-05,3.290350820582197e-11,2.2234614865425384e-21,161
10^{13},-0.010288800509017655,0.0999332479105502,6.67520894497842E-05,3.2907793243858174e-12,0.0,181
10^{14},-0.010288800380904666,0.0999332479105502,6.67520894497842E-05,3.2910642527167466e-13,-3.1763735522036263e-22,201
10^{15},-0.010288800367875843,0.0999332479105502,6.67520894497842E-05,3.2890790785389257e-14,1.5881867761018131e-22,222
10^{16},-0.010288800366549392,0.0999332479105502,6.67520894497842E-05,3.289198340777899e-15,-3.22931311140702e-21,242
10^{17},-0.010288800366414386,0.0999332479105502,6.67520894497842E-05,3.2981430087009045e-16,-2.4352197233561135e-21,259
10^{18},-0.010288800366396037,0.0999332479105502,6.67520894497842E-05,7.296680620827445e-17,-1.2176098616780567e-21,265
10^{19},-0.010288800366396927,0.0999332479105502,6.67520894497842E-05,4.255493527005605e-17,-8.470329472543003e-22,269
10^{20},-0.01028880036640388,0.0999332479105502,6.67520894497842E-05,8.256199543477116e-17,-3.017554874593445e-21,263
10^{21},-0.010288800366399283,0.0999332479105502,6.67520894497842E-05,5.095750210681871e-18,-1.5352472168984194e-21,288
10^{22},-0.010288800366399285,0.0999332479105502,6.67520894497842E-05,5.2583805365546965e-18,-1.2705494208814505e-21,295
10^{23},-0.010288800366399278,0.0999332479105502,6.67520894497842E-05,5.2583805365546965e-18,-1.2705494208814505e-21,283
10^{24},-0.010288800366399278,0.0999332479105502,6.67520894497842E-05,5.2583805365546965e-18,-2.0117032497289633e-21,289
10^{25},-0.01028880036639928,0.0999332479105502,6.67520894497842E-05,4.607859233063394e-18,-2.6999175193730823e-21,292
\end{filecontents*}
\pgfplotstabletypeset[
	col sep = comma,
	header = false,
	skip rows between index = {0}{1},
	skip rows between index = {2}{4},
	skip rows between index = {5}{7},
	skip rows between index = {8}{10},
	skip rows between index = {11}{13},
	skip rows between index = {14}{16},
	skip rows between index = {17}{19},
	skip rows between index = {20}{22},
	skip rows between index = {23}{25},
	sci,
	sci zerofill,
	column type = l,
	columns={0,1,3,4,5,6},
	every head row/.style = {
		before row = \toprule,
		after row = \midrule,
	},
	every last row/.style = {
		after row = \bottomrule,
	},
	columns/0/.style = {
		column type = l,
		column name = {$\rho$},
		string type,
		assign cell content/.code = {
			\pgfkeyssetvalue{/pgfplots/table/@cell content}{$##1$}
		},
	},
	columns/1/.style = {
		column name = {$f(\tu^*)$},
		dec sep align,
	},
	columns/2/.style = {
		column name = {$\norm{\nabla u_h^*}_0$},
		dec sep align,
	},
	columns/3/.style = {
		column name = {$K-\norm{\nabla u_h^*}_0$},
		dec sep align,
	},
	columns/4/.style = {
		column name = {$\psi(\tu^*)$},
		dec sep align,
	},
	columns/5/.style = {
		column name = {$\phi(\tu^*)$},
		dec sep align,
	},
	columns/6/.style = {
		column name = {DC},
		column type = r,
		string type,
	},
	]{penalty.csv}
\end{table}

\Cref{table:penalty-parameter} shows the results for various values of $\rho$.
We clearly observe that the feasibility improves with growing values of $\rho$.
Note that the increase in DC~iterations is caused by a different stopping criterion terminating the algorithm when
\begin{equation}
	\label{eq:stopping-criterion:alternative}
	\frac{\abs[big]{\psi(\sequence{\tu}{k+1}) - \psi(\sequence{\tu}{k})}}{\psi(\sequence{\tu}{k})}
	<
	10^{-3}
	,
\end{equation}
and $\sequence{K}{k} = K$, \cf \eqref{eq:K-schedule}, which we employ to react to an observed stagnation of the feasibility measure $\psi(\sequence{\tu}{k})$; see \cref{figure:penalty-parameter}.
Alternatively, an upper bound on the number of DC iterations after the scheduling has driven $\sequence{K}{k}$ to $K$ would have achieved similar results.

\begin{figure}[htp]
	\centering
	\includegraphics[width = 0.8\linewidth]{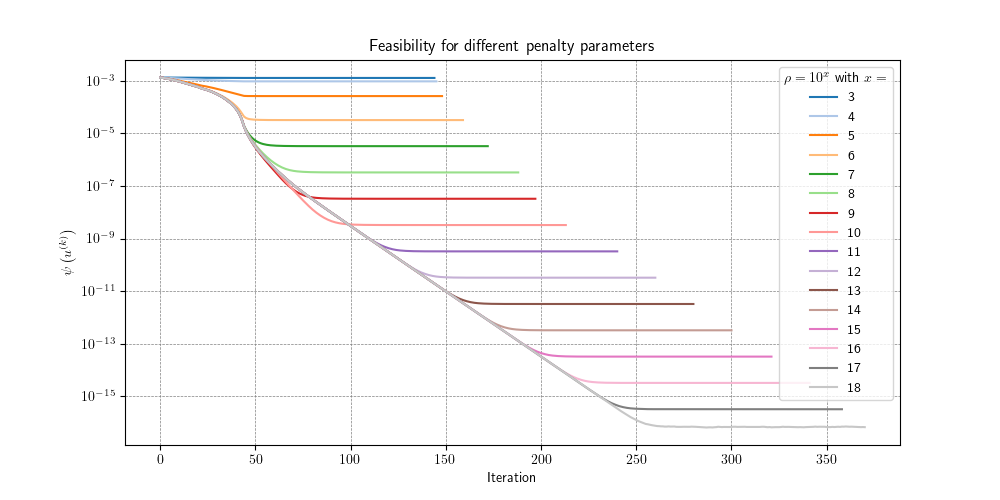}
	\caption{%
		Value of $\psi\sequence[big](){\tu}{k}$ over the iterations for different values of $\rho$.
		After the stopping criterion \eqref{eq:stopping-criterion:alternative} is satisfied, we let the algorithm execute another $100$~iterations.
	}
	\label{figure:penalty-parameter}
\end{figure}

An illustration of the results for smaller penalty parameters can be seen in \cref{figure:solution:256:penalty-parameter}.
Clearly, sparsity in the solution's gradient increases with increasing $\rho$, where $\rho = 10^7$ produces a solution that already captures the defining characteristics of the solution for $\rho = 10^{18}$, \cf \cref{figure:solution:512}.
Nevertheless, in order to achieve feasibility up to machine precision we choose $\rho = 10^{18}$.

\begin{figure}[htbp]
	\centering
	\includegraphics[width = 0.45\linewidth]{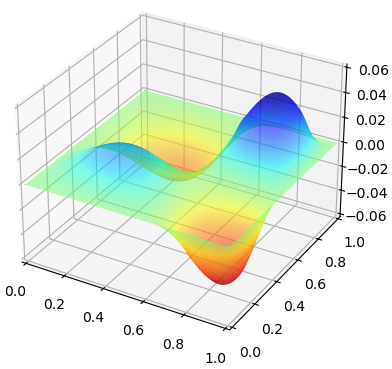}
	\hfill
	\includegraphics[width = 0.45\linewidth]{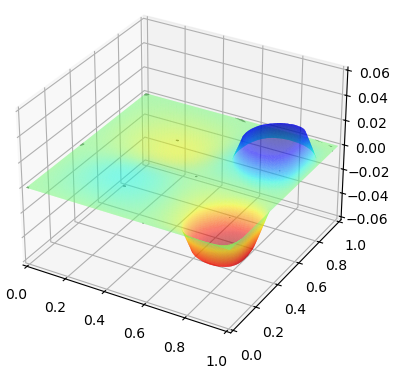}
	\\
	\includegraphics[width = 0.45\linewidth]{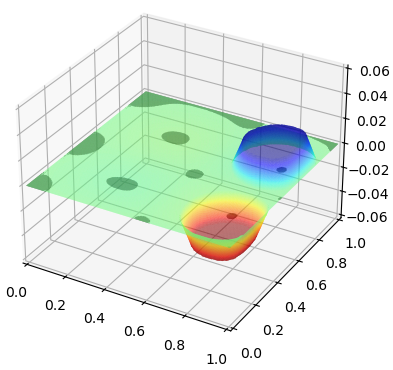}
	\hfill
	\includegraphics[width = 0.45\linewidth]{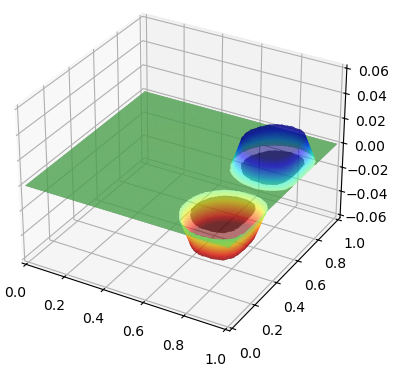}
	\caption{%
		Computed solutions for $\rho = 10^4, 10^5, 10^6, 10^7$ (top left to bottom right).
		Triangles where the solution is constant in the sense of our threshold are indicated by a darker shade.
	}
	\label{figure:solution:256:penalty-parameter}
\end{figure}

\subsection{Optimal Control Problem}
\label{subsection:optimal-control-problem}

In this section, we additionally consider the exemplary optimal control problem
\begin{equation}
	\begin{aligned}
		\text{Minimize}
		\quad
		&
		\frac{1}{2} \norm{y - y_d}_{L^2(\Omega)}^2
		+
		\frac{\alpha_1}{2} \norm{u}_{L^2(\Omega)}^2
		+
		\frac{\alpha_2}{2} \norm{\nabla u}_{L^2(\Omega)}^2
		+
		\frac{\alpha_3}{2} \sum_{i=1}^2 \abs{\frac{\partial u}{\partial x_i}}_{H^\varepsilon(\Omega)}^2
		\\
		\text{\st}
		\quad
		&
		-\Delta y = u \text{ on } \Omega
		,
		\quad
		y = 0 \text{ on } \partial \Omega
		,
		\quad
		\norm{\nabla u}_0 \le K
		\\
		\text{and}
		\quad
		&
		u,y \in U
		,
	\end{aligned}
\end{equation}
where we choose again $\Omega \coloneqq \interval(){0}{1}^2$ and $U \coloneqq H^1_0(\Omega) \cap H^{1+\varepsilon}(\Omega)$ with $\varepsilon \coloneqq \frac14$.
Furthermore, we set $\alpha_1$ and $\alpha_2$ to $10^{-7}$ and use the target state
\begin{equation*}
	y_d(x,y)
	\coloneqq
	\frac16 \, \sin(4 \, \pi \, x) \, \sin(4 \, \pi \, y) \, \exp(2 \, x)
	.
\end{equation*}

\Cref{table:fractional:optimal-control} presents results for different values of the fractional regularization parameter $\alpha_3$.
Similarly to \cref{subsection:numerical-results:energy-optimization} the results suggest that $\alpha_3 = 0$ can be used.

\begin{table}[htp]
	\centering
	\caption{Results for different target mesh sizes~$h$ and regularization parameter~$\alpha_3$: tracking-term,  deviation of the cardinality constraint ($K = 0.2$), feasibility and number of DC iterations.}
	\label{table:fractional:optimal-control}
	\begin{filecontents*}{optcon_discretization_frac.csv}
1/h,gamma,Tracking-term,function value,function value without Frac-term,L0-norm,K-L0,1-feas.,2-feas.,DC
32,$10^{-5}$,0.2954317414597887,0.043642260956548225,0.04363995721956799,0.19983636892137294,0.00016363107862706,9.523913775677517e-14,-3.3881317890172014e-21,204
,$10^{-6}$,0.2953390477809844,0.043635158908435036,0.0436126053422689,0.1998105582295251,0.0001894417704749,9.797024302926616e-14,-2.168404344971009e-19,198
,$10^{-7}$,0.2944433668361052,0.04356698377611148,0.04335117711265037,0.19983653905653231,0.00016346094346769,9.281464485866309e-14,-2.7755575615628914e-17,217
,$10^{-8}$,0.28754314874503545,0.04288116861341672,0.04151462686945925,0.1991058835718955,0.0008941164281045,8.992806499463768e-14,-6.217248937900877e-15,253
,$10^{-9}$,0.26925024281633275,0.04005693992172281,0.03835629260412492,0.1990529646143145,0.0009470353856855,9.85878045867139e-14,-2.1316282072803006e-14,211
,$10^{-10}$,0.2607915578262686,0.03820897798463543,0.03789741979081955,0.1997482330867772,0.0002517669132228,9.237055564881302e-14,9.947598300641403e-14,221
,0,0.2598405069425029,0.037958408655661136,0.037958408655661136,0.1999247548657162,0.0000752451342838,9.414691248821327e-14,-1.1368683772161603e-13,227
64,$10^{-5}$,0.29977665236135986,0.04493578761320041,0.04493302102917531,0.19997303822869583,0.00002696177130417,8.718394929607687e-14,-4.658681209898652e-21,140
,$10^{-6}$,0.29966707539497467,0.04492764528096914,0.044900215693674445,0.19984579328865876,0.00015420671134124,9.458059335720748e-14,2.168404344971009e-19,150
,$10^{-7}$,0.29860958692700856,0.044847240579957336,0.044587388927337296,0.19987803207917215,0.00012196792082785,9.28285226464709e-14,-4.5102810375396984e-17,180
,$10^{-8}$,0.2907390123485496,0.04411260021193996,0.04248110141457489,0.19988946413684427,0.00011053586315573,8.959499808725013e-14,3.1086244689504383e-15,208
,$10^{-9}$,0.27120042319853527,0.040944109953522836,0.03915205165776935,0.19991279102336762,0.00008720897663238,9.769962616701378e-14,-1.7763568394002505e-14,218
,$10^{-10}$,0.26276919802864257,0.03902036696782265,0.03870368380917632,0.19993059529283888,0.00006940470716112,9.325873406851315e-14,-2.1316282072803006e-14,212
,0,0.26144444887977386,0.038693155861355154,0.038693155861355154,0.19990437854820353,0.00009562145179647,9.50350909079134e-14,-4.263256414560601e-14,210
128,0,0.261712361027409,0.03887461017380774,0.03887461017380774,0.19995475581773767,0.00004524418226233,9.85878045867139e-14,-1.4210854715202004e-14,276
256,0,0.26170657883627685,0.038912069901288156,0.038912069901288156,0.1999694116349811,0.0000305883650189,8.659739592076221e-14,1.9095836023552692e-14,177
512,0,0.2615995712880614,0.038911203386735115,0.038911203386735115,0.19999927177354412,7.2822645588e-7,9.969802761133906e-14,3.6637359812630166e-15,211
\end{filecontents*}
\pgfplotstabletypeset[
	col sep = comma,
	header = false,
	skip rows between index = {0}{1},
	sci,
	sci zerofill,
	column type = l,
	columns = {0,1,2,6,7,8,9},
	every head row/.style = {
		before row = \toprule,
		after row = \midrule,
	},
	every row no 7/.style = {
		before row = \midrule,
	},
	every row no 14/.style = {
		before row = \midrule,
	},
	every row no 15/.style = {
		before row = \midrule,
	},
	every row no 16/.style = {
		before row = \midrule,
	},
	every last row/.style = {
		after row = \bottomrule,
	},
	columns/0/.style = {
		column name = {$1/h$},
		column type = r,
		string type,
	},
	columns/1/.style = {
		column name = {$\alpha_3$},
		column type = r,
		string type,
	},
	columns/2/.style = {
		column name = {$\norm{y^* - y_d}_{L^2(\Omega)}$},
		dec sep align,
	},
	columns/3/.style = {
		column name = {$f(\tu^*)$},
		dec sep align,
	},
	columns/4/.style = {
		column name = {$f(\tu^*)-\frac{\alpha_3}{2}(\tu^*)^\transp \tL \tu^*$},
		dec sep align,
	},
	columns/5/.style = {
		column name = {$\norm{\nabla \tu^*}_0$},
		dec sep align,
	},
	columns/6/.style = {
		column name = {$K-\norm{\nabla u_h^*}_0$},
		dec sep align,
	},
	columns/7/.style = {
		column name = {$\psi(\tu^*)$},
		dec sep align,
	},
	columns/8/.style = {
		column name = {$\phi(\tu^*)$},
		dec sep align,
	},
	columns/9/.style = {
		column name = {DC},
		column type = r,
		string type,
	},
	]{optcon_discretization_frac.csv}
\end{table}

The computed control and associated state for $\alpha_3 = 0$ as well as the target state $y_d$ are depicted in \cref{figure:optimal-control}.
Here, we used $1/h = 512$, $\rho = 10^{18}$, $K = 0.2$ and the stopping criterion \eqref{eq:stopping-criterion} with a tolerance of $10^{-13}$.

\begin{figure}[htbp]
	\centering
	\includegraphics[width = 0.45\linewidth]{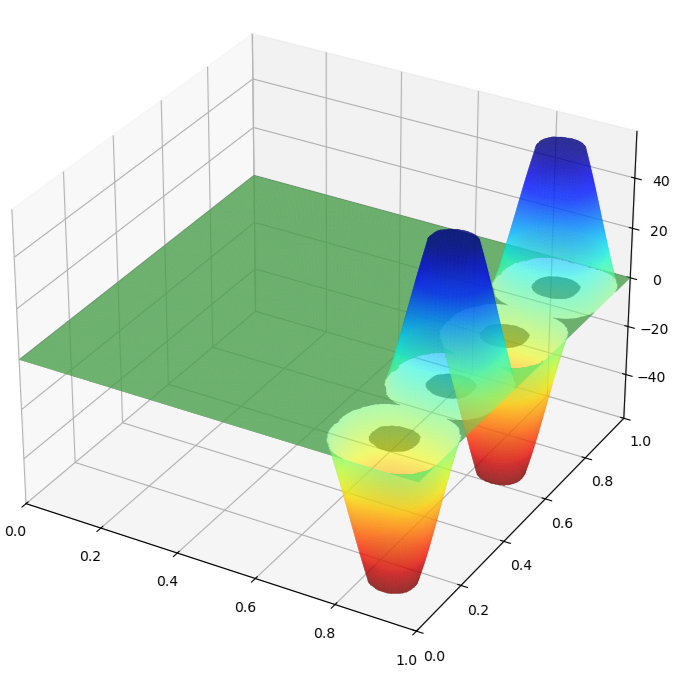}
	\hfill
	\includegraphics[width = 0.45\linewidth]{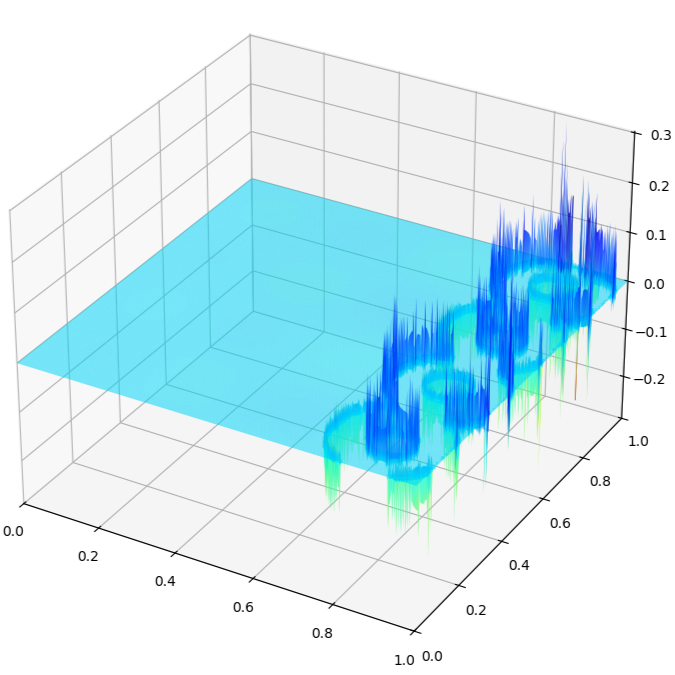}
	\\
	\includegraphics[width = 0.45\linewidth]{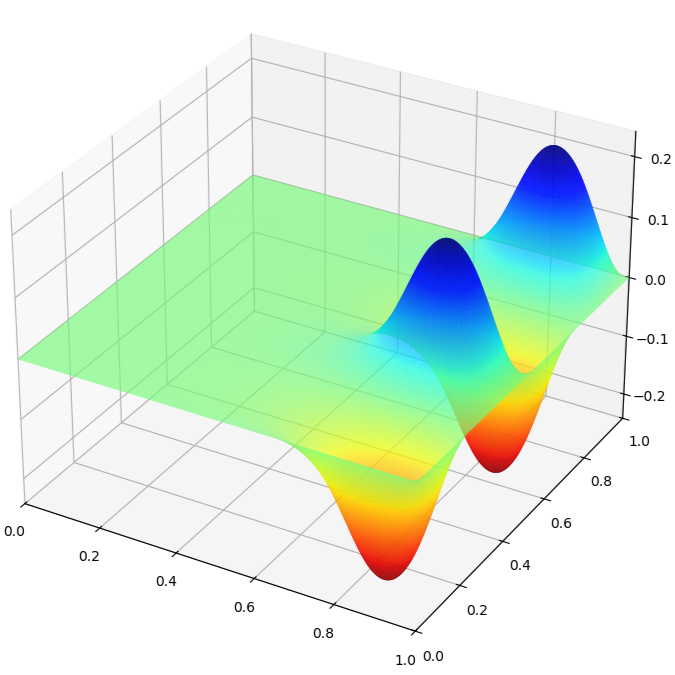}
	\hfill
	\includegraphics[width = 0.45\linewidth]{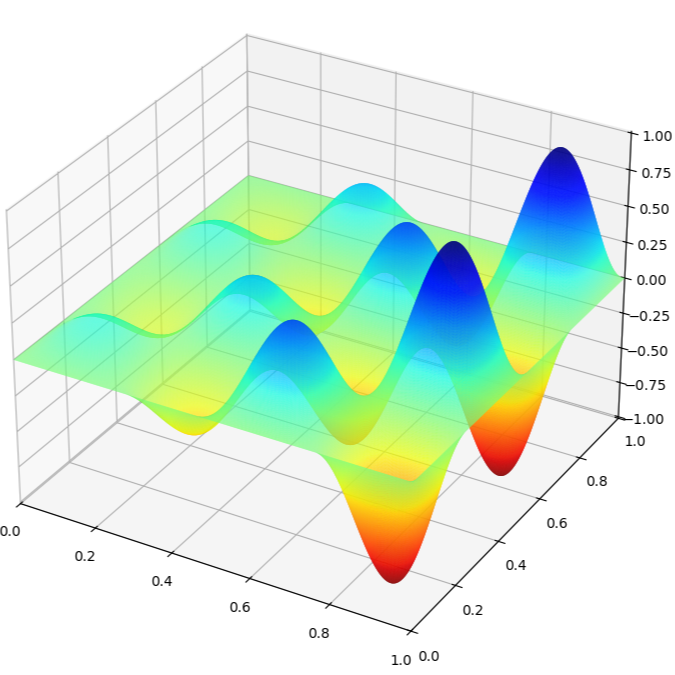}
	\caption{%
		Top left to bottom right: Computed optimal control, $\lambda$, \cf \eqref{eq:general-problem:discrete:penalized:multiplier}, computed optimal state and target state $y_d$.
		Triangles where the computed optimal control is constant in the sense of our threshold are indicated by a darker shade.
	}
	\label{figure:optimal-control}
\end{figure}

\subsection{Image Reconstruction Problem}
\label{subsection:image-reconstruction-problem}

Lastly, we want to show the potential of the \namedref{algorithm:DC-algorithm}{DC Algorithm} in an image reconstruction application based on a discrete formulation of the Potts problem, namely the formulation from \cite[Equation (2)]{StorathWeinmannFrikelUnser:2015:1}, but with an explicit constraint instead of a penalization.
Being an image reconstruction problem, this does not perfectly match the framework presented up to here, but the modifications to the framework required to make this problem treatable are minimal.
Essentially, in the quantity that measures the number of changes in value levels $\tw$, the matrix $\tD_1$, which measured functional value differences, now needs to measure pixel value differences, and the matrix $\tD_2$ that was used to bundle the differences on a per-simplex basis works on color channel bundling.
We end up with a problem of the structure
\begin{equation}
	\begin{aligned}
		\text{Minimize}
		\quad
		&
		\frac{1}{2} \norm{Au - f}_2^2
		\quad
		\text{where }
		u \in \R^{m \times n \times c}
		\\
		\text{\st}
		\quad
		&
		\sum_{s=1}^S \omega_s \norm{\nabla_{p_s} u}_0 \le K
		.
	\end{aligned}
\end{equation}
Here, $c = 1$ for gray-scale and $c = 3$ for color images, and $A$ acts on each color channel separately.

As in \cite{StorathWeinmannFrikelUnser:2015:1}, the $L^0$-term depends on a discrete neighborhood system $\cN$, \eg,
\begin{equation}
	\begin{aligned}
		\cN_0
		&
		\coloneqq
		\set{(1,0), (0,1)}
		,
		\\
		\cN_1
		&
		\coloneqq
		\set{(1,0), (0,1), (1,1), (1,-1)}
		,
	\end{aligned}
\end{equation}
with $S \coloneqq \abs{\cN}$, $p_s$ being the elements of $\cN$ and $\omega_s \in \R_{>0}$ being weights.
Moreover, $\norm{\nabla_p u}_0$ counts the non-zero differences according to the displacement vector $p \in \cN$, \ie,
\begin{equation}
	\begin{aligned}
		\norm{\nabla_p u}_0
		&
		\coloneqq
		\abs{\setDef{(i,j)}{u_{i,j} \neq u_{i+p_1, j+p_2}}}
		.
	\end{aligned}
\end{equation}
For color images, the pixel values $u_{i,j} \in \R^3$ are compared as vectors, \ie, a difference is counted if the two pixels differ in at least one color channel, see \cite[Section 3.3]{StorathWeinmann2014}.

We then have
\begin{equation}
	\begin{aligned}
		\sum_{s=1}^S \omega_s \norm{\nabla_{p_s} u}_0
		=
		\norm{\sumOfSquaredDifferences{\tu}}_{0,\nu}
		=
		\norm{\tD_2 \paren[big][]{(\tD_1 \tu) \odot (\tD_1 \tu)}}_{0,\nu}
	\end{aligned}
\end{equation}
with $\tD_1$ defined such that $\tD_1 \tu$ contains all required differences depending on the neighborhood system $\cN$, $\tD_2$ simply being the identity matrix in the case of a gray-scale image or bundling the rgb-channel value differences into a single entry, and $\nu$ be defined accordingly containing the respective weights $\omega_s$.
This setting is not covered by the finite element framework of \cref{subsection:discretization-constraint-reformulation-existence} as there is no mesh and, accordingly, \cref{assumption:abstract-w-properties}~\ref{assumption:abstract-w-properties:indicator} has no meaning.
Instead, the quantity $\nvertices$ is the number of pixel values, and the entries of $\tw(\tu)$, indexed by~$i$, correspond to the pixel pairs in a stencil defined by the discrete neighborhood system.
The index sets $\constantsimplices{\tu}$ and $\nonconstantsimplices{\tu}$ are still characterized by the zeros of $\tw(\tu)$, \ie, $i \in \constantsimplices{\tu}$ if and only if the pixel values of pair~$i$ coincide in all color channels, and $\mu(\Omega)$ has to be replaced by $\norm{\nu}_1$ in the range of~$K$.
Since $\tD_2$ has nonnegative entries, \cref{assumption:w-matrix-structure} is satisfied, and $\tw$ is continuous.
Only these properties of $\tw$ are used in \cref{lemma:convergent-nodal-values-detect-non-constant-cells}, \cref{theorem:general-problem:discrete:existence,theorem:general-problem:discrete:penalized:convergence}, the results of \cref{subsection:DC-reformulation:DC-approach} up to \cref{theorem:DC-algorithm:convergence}, and \cref{corollary:discrete-optimality-conditions:penalized:exact-penalization}.
Hence, provided that the objective satisfies \cref{assumption:standing-assumptions:finite-dimensions:differentiability-and-convexity}, the discrete problem possesses a global minimizer, global minimizers of the penalized problems converge as $\rho \to \infty$, and \cref{algorithm:DC-algorithm} has the convergence properties stated in \cref{theorem:DC-algorithm:convergence}.
In contrast, the results of \cref{subsection:optimality-conditions}, \cref{lemma:constant-functions-yield-zero-sum-of-squared-differences} and \cref{theorem:discrete-optimality-conditions:penalized} rely on the finite element mesh and do not carry over directly.

We present our results for the specific setting of \cite[Section 4.2]{StorathWeinmannFrikelUnser:2015:1}, where $A$ is a convolutional operator defined by a Gaussian kernel with standard deviation $2$ (in units of pixels) and $f$ is a color image blurred by $A$ and corrupted by additive Gaussian noise of relative level $10^{-1}$, \ie, the noise vector $\eta$ satisfies $\norm{\eta}_2 = 10^{-1} \, \norm{A u^\dagger}_2$ for the true image~$u^\dagger$.
The constraint value for $K$ was set relative to the total weight $\norm{\nu}_1$ of all pixel pairs, namely $K = 0.1 \, \norm{\nu}_1$, $K = 0.05 \, \norm{\nu}_1$ and $K = 0.025 \, \norm{\nu}_1$, and we show iteration \num{100000} for qualitative reference, which is still far from fully converged but allows for a visual comparison.
The true image, the data and the computed solutions are shown in \cref{figure:image-reconstruction}.

\begin{figure}[htbp]
	\centering
	\includegraphics[width = 0.32\linewidth]{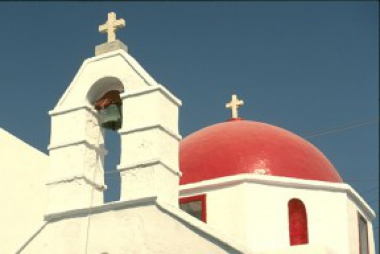}
	\hfill
	\includegraphics[width = 0.32\linewidth]{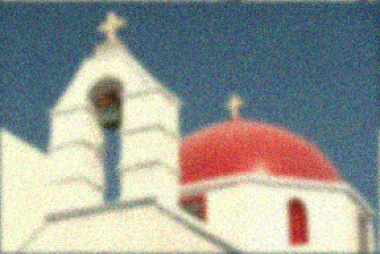}
	\hfill
	\hspace*{0.32\linewidth}
	\\[\medskipamount]
	\includegraphics[width = 0.32\linewidth]{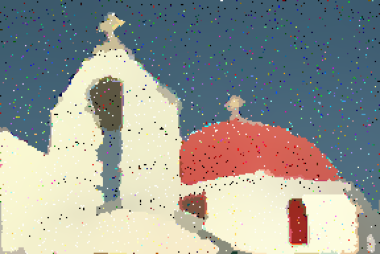}
	\hfill
	\includegraphics[width = 0.32\linewidth]{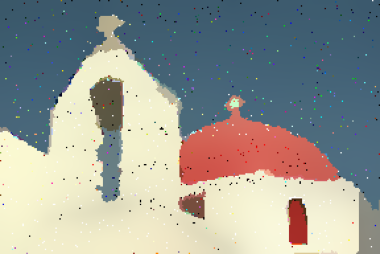}
	\hfill
	\includegraphics[width = 0.32\linewidth]{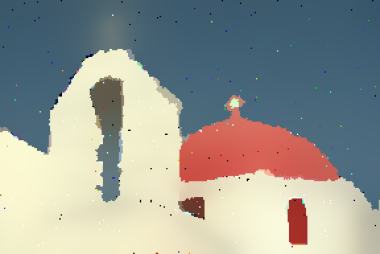}
	\caption{%
		Image reconstruction for the discrete Potts problem with the neighborhood system~$\cN_0$.
		Top row: true image ($192 \times 288$ pixels, three color channels) and the data~$f$ obtained by blurring with~$A$ and adding noise.
		Bottom row: the iterates $\sequence{\tu}{k}$ of \cref{algorithm:DC-algorithm} after $k = \num{100000}$ iterations for the relative constraint values $K = 0.1 \, \norm{\nu}_1$, $K = 0.05 \, \norm{\nu}_1$ and $K = 0.025 \, \norm{\nu}_1$ (left to right).
	}
	\label{figure:image-reconstruction}
\end{figure}

\section*{Data and Code Availability}

The code associated with the numerical results will be made available in a public repository upon acceptance of the paper.
No data was generated in this study.

\section*{Acknowledgments}

Parts of this paper were written while the third author was visiting the University of British Columbia, Vancouver.
He would like to thank the Department of Computer Science for their hospitality.

\appendix
\section{Notation}
\label{section:notation}

\begin{longtable}{p{0.20\linewidth}p{0.74\linewidth}}
	\toprule
	\endfirsthead
	\midrule
	\endhead
	\multicolumn{2}{l}{Definitions of various norms (\cref{section:notation-preliminaries})}
	\\
	\midrule
	$\norm{\,\cdot\,}_0$
	&
	$\ell_0$-pseudo-norm
	\\
	$\norm{\,\cdot\,}_{0,\nu}$
	&
	weighted $\ell_0$-pseudo-norm
	\\
	$\norm{\,\cdot\,}_1$
	&
	$\ell_1$-norm
	\\
	$\norm{\,\cdot\,}_{1,\nu}$
	&
	weighted $\ell_1$-norm
	\\
	$\largestKnorm{\cdot}{K}$
	&
	largest-$K$-norm
	\\
	$\largestKnorm{\cdot}{K}[\nu]$
	&
	weighted largest-$K$-norm
	\\
	\midrule
	\multicolumn{2}{l}{Assumptions on the domain and function spaces (\cref{assumption:standing-assumptions})}
	\\
	\midrule
	$d \in \N$
	&
	spatial dimension
	\\
	$\Omega \subseteq \R^d$
	&
	domain
	\\
	$(\Omega, \cA, \mu)$
	&
	Lebesgue measure space
	\\
	$U$
	&
	space for the unknown
	\\
	\midrule
	\multicolumn{2}{l}{Definition of the gradient-constrained optimization problem \eqref{eq:general-problem} (\cref{section:DC-reformulation:function-space})}
	\\
	\midrule
	$f$
	&
	objective function \eqref{eq:general-problem}
	\\
	$K \in \R_{\geq 0}$
	&
	upper bound for support size \eqref{eq:general-problem}
	\\
	$f_\rho$
	&
	penalized objective function
	\eqref{eq:general-problem:penalized}
	\\
	$\rho \in \R_{> 0}$
	&
	penalty parameter \eqref{eq:general-problem:penalized}
	\\
	$\phi$
	&
	feasibility penalty term \eqref{eq:general-problem:penalized}
	\\
	\midrule
	\multicolumn{2}{l}{Mesh-related quantities (\cref{subsection:discretization-constraint-reformulation-existence})}
	\\
	\midrule
	$\cS$
	&
	mesh, a collection of simplices of dimensions $0, 1, \ldots, d$
	\\
	$\nvertices \in \N$
	&
	number of vertices, indexed as $\ell = 1, \ldots, \nvertices$
	\\
	$\nedges \in \N$
	&
	number of edges, indexed as $j = 1, \ldots, \nedges$
	\\
	$\nsimplices \in \N$
	&
	number of $d-$simplices (cells), indexed as $i = 1, \ldots, \nsimplices$
	\\
	$S_i \in \cS$
	&
	$d$-simplex element (cell), $i = 1, \ldots, \nsimplices$
	\\
	\midrule
	\multicolumn{2}{l}{Quantities related to the discretization \eqref{eq:general-problem:discrete} of problem \eqref{eq:general-problem} (\cref{subsection:discretization-constraint-reformulation-existence})}
	\\
	\midrule
	$\edgetovertices{j}$
	&
	set of vertices $\ell$ incident to edge $j$
	\\
	$\simplextovertices{i}$
	&
	set of vertices $\ell$ that are vertices of simplex $i$
	\\
	$\vertextoedges{\ell}$
	&
	set of edges $j$ incident to vertex $\ell$
	\\
	$\simplextoedges{i}$
	&
	set of edges $j$ that are edges of simplex $i$
	\\
	$\vertextosimplices{\ell}$
	&
	set of simplices $i$ incident to vertex $\ell$
	\\
	$\edgetosimplices{j}$
	&
	set of simplices $i$ incident to edge $j$
	\\
	$U_h$
	&
	finite element space of piecewise affine, cont.\ functions
	\eqref{eq:finite-element-space}
	\\
	$\varphi_\ell$
	&
	nodal basis function, $\ell = 1, \ldots, \nvertices$
	\eqref{eq:finite-element-space}
	\\
	$\tu \in \R^\nvertices$
	&
	coefficient vector representing $u_h \in U_h$
	\\
	$\tD_1 \in \R^{\nedges \times \nvertices}$
	&
	matrix for value jump computation (representing vertex to edge relation in FE applications)
	\\
	$\tD_2 \in \R^{\nsimplices \times \nedges}$
	&
	matrix for bundling value jumps (representing edge to simplex relation in FE applications)
	\\
	$\tw(\tu) \in \R^{\nsimplices}$
	&
	proxy quantity for measuring jumps (vertex value differences in FE applications)
	\\
	$\constantsimplices{\tu}$
	&
	index set of cells where $\tu$ is constant
	\\
	$\nonconstantsimplices{\tu}$
	&
	index set of cells where $\tu$ is non-constant
	\\
	$\nu \in \R^{\nsimplices}_{>0}$
	&
	weight vector with $\nu_i = \mu(S_i)$ \eqref{eq:cell-volumes-are-weights}
	\\
	$F$
	&
	feasible set
	\eqref{eq:general-problem:discrete:feasible-set}
	\\
	\midrule
	\multicolumn{2}{l}{Quantities related to the optimality conditions of problem \eqref{eq:general-problem:discrete} (\cref{subsection:optimality-conditions})}
	\\
	\midrule
	$\tangentCone{F}{\tu}$
	&
	tangent cone to~$F$ at~$\tu$
	\\
	$\td \in \tangentCone{F}{\tu}$
	&
	tangent direction
	\\
	$V \subseteq \{1, \ldots, \nvertices\}$
	&
	subset of vertices
	(\cref{definition:subsets-of-vertices})
	\\
	$\coveredsimplices{V}$
	&
	set of all cells covered by $V$
	(\cref{definition:subsets-of-vertices})
	\\
	$\cV(\tu)$
	&
	collection of certain subsets of vertices
	(\cref{theorem:B-stationarity:explicit-description})
	\\
	$\te_\ell \in \R^\nvertices$
	&
	$\ell$-th standard basis vector
	\\
	\midrule
	\multicolumn{2}{l}{Quantities related to the discrete DC reformulation (\cref{subsection:DC-reformulation:DC-approach})}
	\\
	\midrule
	$\ts \in \R^\nvertices$
	&
	subgradient from the subdifferential $\partial \largestKnorm{\sumOfSquaredDifferences{\cdot}}{K}[\nu](\cdot)$
	\eqref{eq:general-problem:discrete:DC-subproblem}
	\\
	$I \subseteq \set{1, \ldots, \nsimplices}$
	&
	index set of cells
	(\cref{lemma:norm-of-sum-of-squared-differences-via-matrices})
	\\
	$\tM \in \R^{\nvertices \times \nvertices}$
	&
	symmetric positive semidefinite matrix
	(\cref{lemma:norm-of-sum-of-squared-differences-via-matrices})
	\\
	$\tM_I \in \R^{\nvertices \times \nvertices}$
	&
	symmetric positive semidefinite matrix
	(\cref{lemma:norm-of-sum-of-squared-differences-via-matrices})
	\\
	$\chi_I$
	&
	indicator function of the set~$I$ with values in $\set{0,1}$
	(\cref{lemma:norm-of-sum-of-squared-differences-via-matrices})
	\\
	$\tr \in \R^\nsimplices$
	&
	subgradient from the subdifferential $\partial \largestKnorm{\cdot}{K}[\nu](\sumOfSquaredDifferences{\cdot})$
	(\cref{theorem:subdifferential-of-the-largest-K-norm-of-the-sum-of-squared-differences})
	\\
	$\lambda \in \R^\nvertices$
	&
	negative gradient
	\eqref{eq:general-problem:discrete:penalized:multiplier}
	\\
	\midrule
	\multicolumn{2}{l}{Quantities related to the numerical experiments (\cref{section:numerical-results})}
	\\
	\midrule
	$\ninnervertices \in \N$
	&
	number of interior vertices, indexed as $\ell = 1, \ldots, \ninnervertices$
	\\
	$\ninnersimplices \in \N$
	&
	number of $d-$simplices (cells) with at least one interior vertex, indexed as $i = 1, \ldots, \ninnersimplices$
	\\
	$S_i \in \cS$
	&
	$d$-simplex element (cell), $i = 1, \ldots, \ninnersimplices$
	\\
	$\tA \in \R^{\ninnervertices \times \ninnervertices}$
	&
	stiffness matrix
	\\
	$\tB \in \R^{\ninnervertices \times \ninnervertices}$
	&
	mass matrix
	\\
	$\tW, \tilde{\tW}, \tL$
	&
	assembled fractional Sobolev matrices
	\\
	$\tb \in \R^{\ninnervertices}$
	&
	load vector
	\\
	$\varepsilon \in (0,1)$
	&
	fractional Sobolev regularity
	\\
	$\alpha, \alpha_i$
	&
	regularization parameters in the cost functionals
	\\
	$\diam(S_i)$
	&
	diameter of cell $S_i$
	\\
	$\psi$
	&
	alternative feasibility penalty term used as stopping criterion
	\\
	$\beta$
	&
	scheduling parameter for the $K$-schedule
	\eqref{eq:K-schedule}
	\\
	$h$
	&
	target mesh size
	\\
	\bottomrule
\end{longtable}

\printbibliography

\end{document}